\documentclass[10pt]{article}
\usepackage{amsfonts}
\usepackage{mathrsfs,enumerate,amssymb,amsmath,color,lineno}
\usepackage{indentfirst}
\usepackage{amsmath}
\usepackage{amssymb}
\usepackage[amsmath,thmmarks]{ntheorem}
\usepackage{graphicx}
\usepackage{tikz}
\usepackage{caption}
\usepackage{subcaption}
\usepackage{caption}
\usepackage{algorithm}
\usepackage{algpseudocode}
\makeatletter
\@addtoreset{algorithm}{section}
\makeatother
\usepackage{booktabs}
\usepackage{hyperref}

\newtheorem{definition}{\bf Definition}[section]
\newtheorem{lemma}{\bf Lemma}[section]
\newtheorem{theorem}{\bf Theorem}[section]
\newtheorem{remark}{\bf Remark}[section]
\newtheorem{corollary}{\bf Corollary}[section]
\newtheorem{example}{\bf Example}[section]

\newtheorem{proposition}{\bf Proposition}[section]

\begin{document}
\setcounter{page}{1}

\title{{\textbf{Bounds of triangular subnorms and their algorithms}}\thanks {Supported by
the National Natural Science Foundation of China (No.12471440)}}
\author{Ting Tang\footnote{\emph{E-mail address}: Tangting199099@cwnu.edu.cn}, Xue-ping Wang\footnote{Corresponding author. xpwang1@hotmail.com; fax: +86-28-84761502},\\
\emph{(School of Mathematical Sciences, Sichuan Normal University,}\\
\emph{Chengdu 610066, Sichuan, People's Republic of China)}}

\newcommand{\pp}[2]{\frac{\partial #1}{\partial #2}}
\date{}
\maketitle
\begin{quote}
{\bf Abstract} This article deals with the upper and lower bounds of triangular subnorms generated by continuous, strictly decreasing additive generators. It first establishes necessary and sufficient conditions for the comparability of such triangular subnorms. It then explores the existence of strict (resp. nilpotent) bounds of a finite family of strict (resp. nilpotent) triangular subnorms generated by continuous, strictly decreasing additive generators. By duality, completely analogous results are derived for triangular superconorms generated by continuous, strictly increasing additive generators. In particular, it supplies the corresponding algorithms for computing those bounds, which are illustrated by several examples.

\end{quote}

\section{Introduction}
Triangular norms (t-norms for short), introduced by Schweizer and Sklar \cite{BS1960} in order to generalize the triangle inequality towards probabilistic metric spaces, are indispensable tools not only in several other areas of mathematics such as many-valued logics, fuzzy set theory, non-additive measure and integral theory, but also in many branches of information science such as fuzzy control, neural networks and multi-criteria decision
aid (see e.g. \cite{PH1998,LS2008,HN2006,MO2003,EP1995,MS1985}).

Before Jenei \cite{SJ2001} who introduced triangular subnorms (t-subnorms, in brief), most research focuses on t-norms. From the practical point of view in information theory, the strict boundary condition of t-norms, which can be treated as a perfect interaction with certain information, is overly restrictive. Therefore, relaxing this condition may enable us to discount the effects in aggregation, but providing a more flexible aggregative framework. T-subnorms exactly formalizes this idea, extending t-norm theory to accommodate a broader range of applications. On the other hand, T-subnorms also play a crucial role in the ordinal-sum based construction of left-continuous t-norms \cite{SJ2002}, as well as in other construction methods \cite{SJ2005}. One can even associate with each t-norm a (usually infinite) set of t-subnorms \cite{KCM2009}, conversely, redefining the boundary values of a t-subnorm always yields a t-norm \cite{EP2000}. In particular, Urba\'{n}ski and Wasowski \cite{MK2005} investigated boundary weak t-norms in the framework of fuzzy arithmetics to sum up fuzzy numbers and Ouyang \cite{YO2007} further provided their construction methods based on additive generators, and see \cite{RG2004,RG2014,KC2009,AM2004} for more details on t-subnorms and their applications.

In some application contexts of t-subnorms such as fuzzy number operations, fuzzy partition and fuzzy equivalence relations, the following practical issue arises: Given a finite family of t-subnorms belonging to a specific class, one needs to find their upper and lower bounds within the same class. This problem becomes more challenging when the supremum or infimum is sought. Till now, these problems remain open except some cases \cite{VM1999,VM2001}. In this article, we focus on the upper and lower bounds of t-subnorms generated by continuous, strictly decreasing additive generators. Note that, unlike continuous Archimedean t-norms, continuous Archimedean t-subnorms are not necessarily generated by additive generators as shown by Example 4.2 of \cite{RG2014}. Meanwhile, a continuous Archimedean t-subnorm with a continuous, strictly decreasing additive generator may have additive generators that are neither continuous nor strictly decreasing as illustrated by Example 2.13 (iv) in \cite{HL2023}. Some continuous Archimedean t-subnorms even don't admit continuous additive generator as shown by Example 4 of \cite{AM2004}. On the other hand, it is well-known that a continuous Archimedean t-norm is either strict or nilpotent, but continuous Archimedean t-subnorms aren't restricted to these two categories. Indeed, there exist continuous Archimedean t-subnorms that are neither strict nor nilpotent, as illustrated by the following example.
\begin{example}\label{exp2.2}\emph{
Let $S:[0,1]^{2}\rightarrow[0,1]$ be given by
\begin{equation*}
S(x,y)=
\begin{cases}
xy & \hbox{if }\ (x,y)\in[0,0.5]^{2},\\
0.25 & \hbox{if }\ (x,y)\in[0.5,1]^{2},\\
0.5{\mbox{min}(x,y)}& \hbox{otherwise.}\,
\end{cases}
\end{equation*}
Then $S$ is a continuous, Archimedean proper t-subnorm that is neither nilpotent nor strict (see Fig. \ref{fig1}).}
\end{example}
\begin{figure}[!h]
\centering
\includegraphics[width=1\textwidth, keepaspectratio]{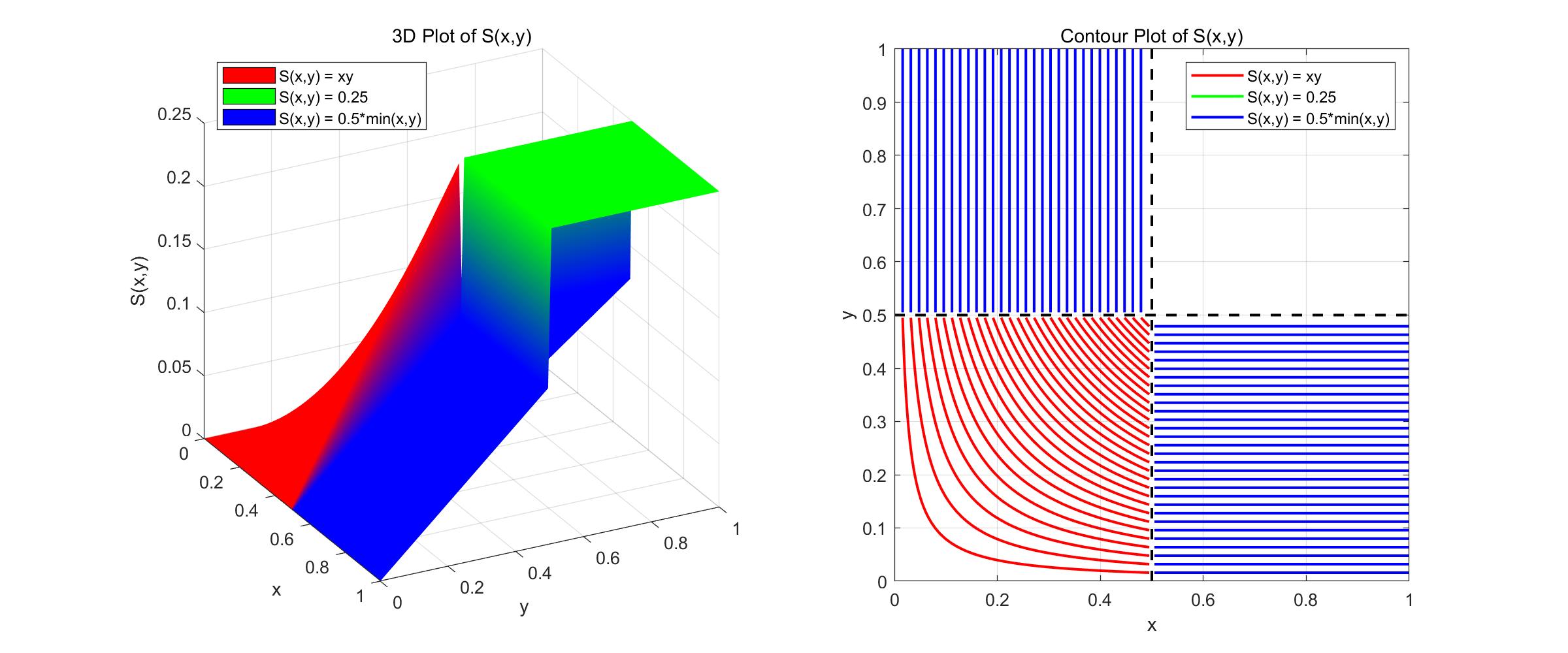}
\caption{3D plot (left) and contour plot of the t-subnorm in Example \ref{exp2.2}.}
\label{fig1}
\end{figure}
However, Proposition 2.14 of \cite{HL2023} shows that a t-subnorm $S$ generated by a continuous, strictly decreasing additive generator is necessarily continuous and Archimedean and, $S$ is either strict or nilpotent.
It is worth pointing out that for two t-subnorms $S_{1}:[0,1]^2\rightarrow[0,1]$ and $S_{2}:[0,1]^2\rightarrow[0,1]$ with continuous, strictly decreasing additive generators, it is trivial that the minimum t-norm $T_{M}:[0,1]^2\rightarrow[0,1]$ with $T_{M}(x,y)=\mbox{min}(x,y)$ is an upper bound of both $S_{1}$ and $S_{2}$, and the zero t-subnorm $Z:[0,1]^2\rightarrow[0,1]$ with $Z(x,y)=0$ is a lower bound of both $S_{1}$ and $S_{2}$, i.e.,
$$Z\leq \mbox{min}(S_{1},S_{2})\leq \mbox{max}(S_{1},S_{2})\leq T_{M}.$$
Therefore, in what follows we mainly consider the following interesting problem: Given two t-subnorms $S_{1}$ and $S_{2}$ with continuous, strictly decreasing additive generators, can we find two respective non-trivial bounds $S_{u}$ and $S_{l}$ with continuous, strictly decreasing additive generators such that
$$S_{l}\leq \mbox{min}(S_{1},S_{2})\leq \mbox{max}(S_{1},S_{2})\leq S_{u}?$$

The rest of this article is organized as follows. In Section 2, we recall some basic concepts and results concerning t-norms and t-subnorms, respectively. In Section 3, we explore the comparison of t-subnorms generated by continuous, strictly decreasing additive generators and establish the necessary and sufficient condition for such comparisons. In Section 4, we demonstrate the existence of a strict (resp. nilpotent) upper bound for any two strict (resp. nilpotent) t-subnorms generated by continuous, strictly decreasing additive generators, and provide a corresponding algorithm to compute such an upper bound. In Section 5, we further discuss the existence of a strict (resp. nilpotent) lower bound for two such strict (resp. nilpotent) t-subnorms, along with algorithm to compute it. A conclusion is drawn in Section 6.

\section{Preliminaries}
In this section, we recall some definitions and results about t-norms and t-subnorms which will be used in the sequel.

A binary operation $T:[0,1]^{2}\rightarrow[0,1]$ is a $t$-$norm$ if it is commutative, associative, non-decreasing in both variables and 1 is its neutral element. A binary operation $C:[0,1]^{2}\rightarrow[0,1]$ is a $t$-$conorm$ if it is commutative, associative, non-decreasing in both variables and 0 is its neutral element \cite{EP2000}.

The duality between t-norms and t-conorms is expressed by the fact that from any t-norm $T$ we can obtain its dual t-conorm $C$ by the equation $C(x,y)=1-T(1-x,1-y)$ and vice versa.

\begin{definition}[\cite{EP2000}]\label{def:2.3}
 \emph{A t-norm $T:[0,1]^{2}\rightarrow[0,1]$ is said to be
 \begin{enumerate}
\renewcommand{\labelenumi}{(\roman{enumi})}
  \item \emph{Archimedean} if for each $(x,y)\in(0,1)^2$ there is an $n\in \mathbb{N}=\{1,2,\cdots, n,\cdots\}$ with $x_{T}^{(n)}<y$, where $x_{T}^{(1)}=x$ and $x_{T}^{(n+1)}=T(x,x_{T}^{(n)})$;
  \item \emph{strict} if it is continuous and $T(x,y)<T(x,z)$ whenever $x>0$ and $y<z$;
  \item \emph{nilpotent} if it is continuous and for each $x\in(0,1)$ there exists some $n\in\mathbb{N}$ such that $x_{T}^{(n)}=0$.
\end{enumerate}}
\end{definition}

It is well known that each continuous Archimedean t-norm (resp. t-conorm) is either strict or nilpotent.

\begin{definition}[\cite{EP1999,EP2000}]\label{def2.4}
\emph{Let $a, b, c, d\in [-\infty, \infty]$ with $a<b, c<d$ and $f:[a,b]\rightarrow[c,d]$ be a non-increasing function. Then the function $f^{(-1)}:[c,d]\rightarrow[a,b]$ defined by
\begin{equation*}
f^{(-1)}(y)=\sup\{x\in [a,b]\mid f(x)>y\}
\end{equation*}
is called a \emph{pseudo-inverse} of the non-increasing function $f$. In particular, $\sup\emptyset=a$.}
\end{definition}
%\begin{definition}
%\emph{An \emph{additive generator} $t :[0,1]\rightarrow [0,\infty]$ of a t-norm $T$ is a strictly decreasing function which is also right-continuous in $0$ and satisfies $t(1) = 0$, such that for all $(x,y)\in [0,1]^2$ we have
%\begin{equation*}
%\label{eq:1.7}
%t(x)+t(y)\in \mbox{Ran}(t)\cup [t(0),\infty],
%\end{equation*}
%\begin{equation*}
%\label{eq:1.8}
% T(x,y)= t^{(-1)}(t(x)+t(y)).
%\end{equation*}}
%\end{definition}

The following is the representation theorem of continuous Archimedean t-norms.
\begin{theorem}[\cite{EP2000}]\label{theorem:2.1}
For a function $T:[0,1]^{2}\rightarrow[0,1]$ the following are equivalent:
\begin{enumerate}
\renewcommand{\labelenumi}{(\roman{enumi})}
\item $T$ is a continuous Archimedean t-norm.
\item $T$ has a continuous additive generator, i.e., there exists a continuous, strictly decreasing function $t:[0,1]\rightarrow[0,\infty]$ with $t(1)=0$, which
is uniquely determined up to a positive multiplicative constant, such that for all $(x,y)\in[0,1]^{2}$
\begin{equation*}
T(x,y) = t^{(-1)}(t(x)+t(y)).
\end{equation*}
\end{enumerate}
\end{theorem}

In Theorem \ref{theorem:2.1}, if $t(0)=\infty$ then $t^{(-1)}$ is the inverse of $t$ and $T$ is strict. If $t(0)<\infty$ then $T$ is nilpotent.

Next we recall the definitions of t-subnorm and its dual t-superconorm, respectively.
\begin{definition}[\cite{SJ2001}]\label{def2.5}
 \emph{A binary operation $S:[0,1]^2\rightarrow[0,1]$ is a $t$-$subnorm$ if it is commutative, associative, non-decreasing in both variables and $S(x,y)\leq \min(x,y)$ for all $(x,y)\in[0,1]^2$.}
\end{definition}

\begin{definition}[\cite{AM2016}]\label{def2.6}
\emph{A binary operation $M:[0,1]^2\rightarrow[0,1]$ is a $t$-$superconorm$ if it is commutative, associative, non-decreasing in both variables and $M(x,y)\geq \max(x,y)$ for all $(x,y)\in[0,1]^2$.}
\end{definition}

Between t-subnorms and t-superconorms there is the same dual relation as between t-norms and t-conorms. Obviously, each t-norm (resp. t-conorm) is also a t-subnorm (resp. t-superconorm). A t-subnorm is said to be \emph{proper} if it is not a t-norm (that is, there exists some $x\in[0,1]$ such that $S(x,1)<x$). In particular, Mesiarov\'{a} \cite{AM2004} proved that a continuous t-subnorm $S$ is proper if and only if $S(1,1)<1$.

The zero t-subnorm $Z$, where $Z(x,y)=0$ for all $(x,y)\in[0,1]^{2}$, is the weakest t-subnorm, i.e., $Z\leq S\leq T_{M}$ for an arbitrary  t-subnorm $S$. In the following, every t-subnorm $S$ with $S\neq Z$ is said to be non-constant.

The following definition is completely analogous to Definition \ref{def:2.3}.
\begin{definition}\label{def2.7}
 \emph{A t-subnorm $S:[0,1]^2\rightarrow[0,1]$ is said to be
 \begin{enumerate}
\renewcommand{\labelenumi}{(\roman{enumi})}
  \item \emph{Archimedean} \cite{AM2016} if for all $x, y\in(0,1)$ there exists an $n\in\mathbb{N}$ such that $x_{S}^{(n)}<y$, where $x_{S}^{(1)}=x$ and $x_{S}^{(n+1)}=S(x,x_{S}^{(n)})$;
  \item \emph{strict} if it is continuous and $S(x,y)<S(x,z)$ whenever $x>0$ and $y<z$;
  \item \emph{nilpotent} if it is continuous and for each $x\in(0,1)$ there exists some $n\in\mathbb{N}$ such that $x_{S}^{(n)}=0$.
\end{enumerate}}
\end{definition}

Observe that the fact that $1$ is the neutral element of t-norms forces the strict monotonicity of any additive generator $t:[0,1]\rightarrow[0,\infty]$ of a t-norm as well as $t(1)=0$. However, for t-subnorms this restriction can be relaxed as follows.
\begin{definition}[\cite{AM2004}]\label{def2.8}
\emph{A non-increasing function $s:[0,1]\rightarrow[0,\infty]$ is called an \emph{additive generator} of a t-subnorm $S:[0,1]^2\rightarrow[0,1]$ if for all $(x,y)\in[0,1]^2$
\begin{equation}\label{eq:1}
S(x,y)=s^{(-1)}(s(x)+s(y)).
\end{equation}}
\end{definition}

Obviously, a continuous and strictly decreasing function $s:[0,1]\rightarrow[0,\infty]$ is an additive generator of a t-subnorm through Eq.\eqref{eq:1}. Moreover, one can easily see that the following basic properties hold.
\begin{proposition}[\cite{HL2023}]\label{pro2.1}
Let $S$ be a t-subnorm with a continuous, strictly decreasing additive generator $s$. Then
\renewcommand{\labelenumi}{(\roman{enumi})}
\begin{enumerate}
\item $S$ is both continuous and Archimedean.
\item $S$ is proper if and only if $s(1)>0$, and $S$ is a t-norm if and only if $s(1)=0$.
\item $S$ is strict if and only if $s(0)=\infty$, and $S$ is nilpotent if and only if $s(0)<\infty$. Hence, $S$ is either strict or nilpotent.
\item $S= Z$ if and only if $s(0)\leq2s(1)$.
\item For all constant $c>0$, $c\cdot s$ is also an additive generator of $S$.
\end{enumerate}
\end{proposition}

Since the order relation of t-subnorms is closely associated with the existence of their upper and lower bounds, we first present, in the following section, the necessary and sufficient conditions for the comparability of two t-subnorms with continuous, strictly decreasing additive generators.
\section{Orders of t-subnorms with continuous, strictly decreasing additive generators}
In this section, we establish a necessary and sufficient condition for the comparison of two t-subnorms with continuous, strictly decreasing additive generators.

From Proposition \ref{pro2.1} it is evident that many fine properties of continuous Archimedean t-subnorms are shown by their continuous, strictly decreasing additive generators. This naturally reminds us to guess that the pointwise comparability of two t-subnorms may also be deducible from their respective continuous, strictly decreasing additive generators. Observe that each additive generator $s:[0,1]\rightarrow[0,\infty]$ of such t-subnorms, viewed as a function from $[0,1]$ to $[s(1),s(0)]$, is a strictly decreasing bijection, and subsequently, its inverse function $s^{-1}: [s(1),s(0)]\rightarrow[0,1]$ always exists.

For convenience, let $\mu:A\rightarrow B$ and $\nu:B\rightarrow C$ be two functions. Define $(\nu\circ\mu)(x)=\nu(\mu(x))$ for any $x\in A$. Then we have the following theorem that is a foundation of forthcoming sections.
\begin{theorem}\label{theorem:3.1}
Let $S_{1}:[0,1]^{2}\rightarrow[0,1]$ and $S_{2}:[0,1]^{2}\rightarrow[0,1]$ be two non-constant t-subnorms with their respective continuous, strictly decreasing additive generators $s_{1},s_{2}:[0,1]\rightarrow[0,\infty]$. Then $S_{1}\leq S_{2}$ if and only if the function $s_{1}\circ s_{2}^{-1}:[s_{2}(1), s_{2}(0)]\rightarrow[s_{1}(1), s_{1}(0)]$ is subadditive, i.e., for all $u,v\in[s_{2}(1), s_{2}(0)]$ with $u+v\in[s_{2}(1), s_{2}(0)]$ we have
\begin{equation}\label{eq:2}
s_{1}\circ s_{2}^{-1}(u+v)\leq s_{1}\circ s_{2}^{-1}(u)+s_{1}\circ s_{2}^{-1}(v).
\end{equation}
\end{theorem}
\begin{proof}
By Eq.(\ref{eq:1}), $S_{1}\leq S_{2}$ if and only if
\begin{equation} \label{eq:3}
s_1^{-1}(\mbox{min}(s_{1}(0), s_1(x)+s_1(y)))\leq s_2^{-1}(\mbox{min}(s_{2}(0), s_2(x)+s_2(y)))
\end{equation} for all $x,y\in[0,1]$.
Let $x,y\in[0,1]$ and set $u=s_{2}(x)$ and $v=s_{2}(y)$. Taking into account that $s_{1}$ is strictly decreasing and applying $s_{1}$ to
both sides of Eq.(\ref{eq:3}), we have that Eq.(\ref{eq:3}) is equivalent to the fact that
\begin{equation} \label{eq:4}
\mbox{min}(s_{1}(0), s_{1}\circ s_{2}^{-1}(u)+s_{1}\circ s_{2}^{-1}(v))\geq s_{1}\circ s_{2}^{-1}(\mbox{min}(s_{2}(0),u+v))
\end{equation}
for all $u,v\in[s_{2}(1), s_{2}(0)]$.

If $S_{1}\leq S_{2}$, then Eq.(\ref{eq:4}) holds for all $u,v\in[s_{2}(1), s_{2}(0)]$ with $u+v\in[s_{2}(1), s_{2}(0)]$. There are two cases as below.
\renewcommand{\labelenumi}{(\roman{enumi})}
\begin{enumerate}
\item If $u+v<s_{2}(0)$, then $s_{1}\circ s_{2}^{-1}(\mbox{min}(s_{2}(0),u+v))=s_{1}\circ s_{2}^{-1}(u+v)$. Thus, by Eq.(\ref{eq:4}), we get
$$s_{1}\circ s_{2}^{-1}(u+v)=s_{1}\circ s_{2}^{-1}(\mbox{min}(s_{2}(0),u+v))\leq s_{1}\circ s_{2}^{-1}(u)+s_{1}\circ s_{2}^{-1}(v).$$
\item If $u+v=s_{2}(0)$, then by Eq.(\ref{eq:4}), Eq.(\ref{eq:2}) hold for all $u,v\in[s_{2}(1), s_{2}(0)]$.
\end{enumerate}

Therefore, in any case, Eq.(\ref{eq:2}) hold for all $u,v\in[s_{2}(1), s_{2}(0)]$, i.e., $s_{1}\circ s_{2}^{-1}$ is subadditive.

Conversely, if the function $s_{1}\circ s_{2}^{-1}:[s_{2}(1), s_{2}(0)]\rightarrow[s_{1}(1), s_{1}(0)]$ is subadditive, then from Eq.(\ref{eq:2}), for all $u,v\in[s_{2}(1), s_{2}(0)]$ with $u+v\in[s_{2}(1), s_{2}(0)]$ we have
$$s_{1}\circ s_{2}^{-1}(\mbox{min}(s_{2}(0),u+v))=s_{1}\circ s_{2}^{-1}(u+v)\leq s_{1}\circ s_{2}^{-1}(u)+s_{1}\circ s_{2}^{-1}(v)$$
and $$s_{1}\circ s_{2}^{-1}(\mbox{min}(s_{2}(0),u+v))\leq s_{1}(0).$$
Therefore, $$s_{1}\circ s_{2}^{-1}(\mbox{min}(s_{2}(0),u+v))\leq\mbox{min}(s_{1}(0), s_{1}\circ s_{2}^{-1}(u)+s_{1}\circ s_{2}^{-1}(v)),$$
i.e., Eq.(\ref{eq:4}) holds. This follows that $S_{1}\leq S_{2}$.
\end{proof}

\begin{remark}\label{remark:3.1} \emph{In Theorem \ref{theorem:3.1}, if both $S_{1}$ and $S_{2}$ are not proper, then they are continuous Archimedean t-norms and their continuous, strictly decreasing additive generators satisfy that $s_{1}(1)=s_{2}(1)=0$ by Proposition \ref{pro2.1} (ii). Consequently, by Theorem \ref{theorem:3.1}, $S_{1}\leq S_{2}$ if and only if the function $s_{1}\circ s_{2}^{-1}:[0, s_{2}(0)]\rightarrow[0, s_{1}(0)]$ is subadditive. Therefore, Theorem \ref{theorem:3.1} generalizes Theorem 6.2 of \cite{EP2000}.}
\end{remark}

The following example illustrates Theorem \ref{theorem:3.1}.
\begin{example}\label{exp3.1}\emph{Consider two t-subnorms $S_{1}:[0,1]^2\rightarrow[0,1]$ with $S_{1}(x,y)=\mbox{max}(x+y-\frac{3}{2},0)$ and $S_{2}:[0,1]^2\rightarrow[0,1]$ with $S_{2}(x,y)=\frac{1}{2}xy$, which are generated by additive generators $s_{1}:[0,1]\rightarrow[0,\infty]$ with $s_{1}(x)=1-\frac{2}{3}x$ and $s_{2}:[0,1]\rightarrow[0,\infty]$ with $s_{2}(x)=-\ln\frac{x}{2}$, respectively. Obviously, both $s_{1}$ and $s_{2}$ are continuous, strictly decreasing. A straightforward calculation shows that the function $s_1 \circ s_2^{-1}$, given by $(s_1\circ s_2^{-1})(u)=1-\frac{4}{3}e^{-u}$, is subadditive for all $u,v\in[\ln2, \infty]$. Therefore, by Theorem \ref{theorem:3.1} we have $S_{1}\leq S_{2}$.}
\end{example}

\section{Upper bounds of t-subnorms with continuous, strictly decreasing additive generators}
This section investigates the existence of an upper bound of both t-subnorms $S_{1}$ and $S_{2}$  with continuous, strictly decreasing additive generators, and an algorithm is given for calculating it.

With Proposition \ref{pro2.1} in hand, for convenience we present the following definition.
\begin{definition}\label{def4.1}
\emph{Let $S:[0,1]^{2}\rightarrow[0,1]$ be a proper t-subnorm with continuous, strictly decreasing additive generator $s:[0,1]\rightarrow[s(1),\infty]$. The function $s_{n}:[0,1]\rightarrow[1,\infty]$ defined by
\begin{equation}\label{eq5}
s_{n}(x)=\frac{s(x)}{s(1)}
\end{equation}
is called a \emph{normalized additive generator} of $S$.}
\end{definition}

Then from Proposition \ref{pro2.1}, we have the following remark.
\begin{remark}\label{remark:4.1} \emph{
\renewcommand{\labelenumi}{(\roman{enumi})}
\begin{enumerate}
\item In Definition \ref{def4.1}, the normalized additive generator $s_{n}$ can be viewed as a strictly decreasing bijection from $[0,1]$ to $[1,s_{n}(0)]$ with $1<s_{n}(0)\leq \infty$. Moreover, from Proposition \ref{pro2.1} (iii), $S$ is strict if and only if $s_{n}(0)=\infty$, and $S$ is nilpotent if and only if $s_{n}(0)<\infty$.
\item In Theorem \ref{theorem:3.1}, if both $S_{1}$ and $S_{2}$ are proper, and $s_{n_{1}}, s_{n_{2}}:[0,1]\rightarrow[1,\infty]$ are their respective normalized additive generators, then $S_{1}\leq S_{2}$ if and only if the function $s_{n_{1}}\circ s_{n_{2}}^{-1}:[1, s_{n_{2}}(0)]\rightarrow[1, s_{n_{1}}(0)]$ is subadditive, where $1<s_{n_{1}}(0), s_{n_{2}}(0)\leq \infty$.
\end{enumerate}}
\end{remark}

It is well-known that each t-subnorm $S:[0,1]^2\rightarrow[0,1]$ can be transformed into a t-norm by redefining (if necessary) its values on the upper right boundary of the unit square, i.e., the following function $T:[0,1]^2\rightarrow[0,1]$ is a t-norm:
\begin{equation*}
T(x,y) =
\begin{cases}
S(x,y) & \text{if } (x,y) \in [0,1)^2, \\
\min(x,y) & \text{otherwise}.
\end{cases}
\end{equation*}
Note that the above t-norm $T$ isn't necessarily continuous even if $S$ is continuous. For example, even if $S:[0,1]^2\rightarrow[0,1]$ with $S(x,y)=\frac{1}{2}xy$ is continuous proper t-subnorm, the following function $T:[0,1]^2\rightarrow[0,1]$ defined by
\begin{equation*}
T(x,y)=
\begin{cases}
\frac{1}{2}xy & \hbox{if }\ (x,y)\in[0,1)^{2},\\
\min(x,y) & \hbox{otherwise},\
\end{cases}
\end{equation*}
is a non-continuous t-norm. But the following result is clearly.
\begin{proposition}\label{pro4.2}
Let $S_{1}:[0,1]^2\rightarrow[0,1]$ and $S_{2}:[0,1]^2\rightarrow[0,1]$ be two non-constant t-subnorms with continuous, strictly decreasing additive generators and, $S_{1}\leq S_{2}$. Then there exists a t-norm $T$ such that
$$T\geq \max(S_{1}, S_{2}).$$
\end{proposition}

On the other hand, the following theorem is well-known.
\begin{theorem}[\cite{EP2000}]\label{thm4.1}
Let $T_{1}:[0,1]^2\rightarrow[0,1]$ and $T_{2}:[0,1]^2\rightarrow[0,1]$ be two continuous Archimedean t-norms. Then there exists a continuous Archimedean t-norm $T$ such that $T\geq \max(T_{1}, T_{2})$.
\end{theorem}

Moreover, we fortunately have the following proposition.
\begin{proposition}\label{pro4.3}
Let $S_{1}:[0,1]^2\rightarrow[0,1]$ and $S_{2}:[0,1]^2\rightarrow[0,1]$ be two non-constant proper t-subnorms with continuous, strictly decreasing additive generators. Then there exists a continuous Archimedean t-norm $T$ such that
$$T\geq \max(S_{1}, S_{2}).$$
\end{proposition}
\begin{proof}Since both $S_{1}$ and $S_{2}$ are generated by continuous, strictly decreasing additive generators, by Proposition \ref{pro2.1} (i) they are continuous, Archimedean proper t-subnorms. Let $s_{n_{1}}$, $s_{n_{2}} :[0,1]\rightarrow[1, \infty]$ be their respective normalized additive generators. Define $t_{i}:[0,1]\rightarrow [0,\infty]$ by $t_{i}(x)=s_{n_{i}}(x)-1$ for $i=1,2$. Then each $t_{i}$ is continuous, strictly decreasing with $t_{i}(1)=0$ due to the properties of $s_{n_{i}}$. Thus, by Theorem \ref{theorem:2.1}, every $t_{i}$ is a continuous additive generator of a continuous Archimedean t-norm $T_{i}$. Furthermore, for $i=1,2$ we have
\begin{align*}
T_{i}(x,y)&=t_{i}^{-1}(\mbox {min}(t_{i}(0), t_{i}(x)+ t_{i}(y)))\\
&=t_{i}^{-1}(\mbox {min}(s_{n_{i}}(0)-1, s_{_{n_{i}}}(x)+ s_{n_{i}}(y)-2))\mbox{ since }t_{i}(x)=s_{n_{i}}(x)-1\\
&=s_{n_{i}}^{-1}(\mbox {min}(s_{n_{i}}(0), s_{_{n_{i}}}(x)+ s_{n_{i}}(y)-1))\mbox{ since }t_{i}^{-1}(x)=s_{n_{i}}^{-1}(x+1)\\
&\geq s_{n_{i}}^{-1}(\mbox {min}(s_{n_{i}}(0), s_{_{n_{i}}}(x)+ s_{n_{i}}(y)))\mbox{ since }s_{n_{i}} \mbox{ is strictly decreasing}\\
&=S_{i}(x,y).
\end{align*}
Consequently, from Theorem \ref{thm4.1} there exists a continuous Archimedean t-norm $T$ such that
$$T\geq \max(T_{1},T_{2})\geq \max(S_{1}, S_{2}).$$
\end{proof}

Moreover, from Proposition \ref{pro4.3} and its proof and Theorem \ref{thm4.1} the following corollary is immediately.
\begin{corollary}\label{coro4.1}
Let $S_{1}:[0,1]^2\rightarrow[0,1]$ and $S_{2}:[0,1]^2\rightarrow[0,1]$ be two non-constant t-subnorms with continuous, strictly decreasing additive generators. If one of them is a proper t-subnorm, then there exists a continuous Archimedean t-norm $T$ such that
$$T\geq \max(S_{1}, S_{2}).$$
\end{corollary}

Proposition \ref{pro4.3} shows that for two proper t-subnorms with continuous, strictly decreasing additive generators, there always exists a continuous Archimedean t-norm serving as an upper bound. Next, we further raise the question: does there also exist a proper continuous Archimedean t-subnorm generated by a continuous, strictly decreasing additive generator serving as an upper bound? In what follows, we positively answer this problem by distinguishing two cases.

\subsection{In the case of proper and strict t-subnorms}
%In the case that both $S_{1}$ and $S_{2}$ are proper and strict, we completely show the existence of an upper bound of $S_{1}$ and $S_{2}$. In addition, we present an algorithm for finding an upper bound of $S_{1}$ and $S_{2}$.

Analogous to the proof of Theorem \ref{theorem:3.1}, the following lemma holds.
\begin{lemma}\label{lemm4.2}
Let $S_{1}:[0,1]^2\rightarrow[0,1]$ and $S_{2}:[0,1]^2\rightarrow[0,1]$ be two non-constant proper and strict t-subnorms with their respective normalized additive generators $s_{n_{1}},s_{n_{2}}:[0,1]\rightarrow[1, \infty]$. Then $S_{1}\leq S_{2}$ if and only if the function $s_{_{n_{2}}}\circ s_{n_{1}}^{-1}:[1,\infty]\rightarrow[1, \infty]$ is superadditive, i.e., for all $u,v\in[1, \infty]$ we have
\begin{equation*}
s_{n_{2}}\circ s_{n_{1}}^{-1}(u+v)\geq s_{n_{2}}\circ s_{n_{1}}^{-1}(u)+s_{n_{2}}\circ s_{n_{1}}^{-1}(v).
\end{equation*}
\end{lemma}

The following example illustrates Lemma \ref{lemm4.2}.
\begin{example}\label{exp4.1}\emph{Let $S_{1}:[0,1]^2\rightarrow[0,1]$ with $S_{1}(x,y)=\frac{1}{2}xy$ and $S_{2}:[0,1]^2\rightarrow[0,1]$ with $S_{2}(x,y)=\frac{xy}{x+y-0.5xy}$, which are generated by normalized additive generators $s_{n_{1}}:[0,1]\rightarrow[1, \infty]$ with $s_{n_{1}}(x)=-\frac{\ln x}{\ln2}+1$ and $s_{n_{2}}:[0,1]\rightarrow[1, \infty]$ with $s_{n_{2}}(x)=\frac{2}{x}-1$, respectively. Obviously, both $S_{1}$ and $S_{2}$ are proper and strict. A simple computation shows that the function $s_{n_{1}} \circ s_{n_{2}}^{-1}:[1,\infty]\rightarrow[1,\infty]$, given by $(s_{n_{1}}\circ s_{n_{2}}^{-1})(u)=\frac{\ln (u+1)}{\ln2}$, is subadditive for all $u,v\in[1, \infty]$ and the function $s_{n_{2}} \circ s_{n_{1}}^{-1}:[1,\infty]\rightarrow[1,\infty]$, given by $(s_{n_{2}}\circ s_{n_{1}}^{-1})(u)=2^{u}-1$, is superadditive for all $u,v\in[1, \infty]$. Therefore, by Theorem \ref{theorem:3.1} or Lemma \ref{lemm4.2} we have $S_{1}\leq S_{2}$.}
\end{example}

Based on Lemma \ref{lemm4.2}, we can establish the following theorem.
\begin{theorem}\label{theorem:4.1}
Let $S_{1}:[0,1]^2\rightarrow[0,1]$ and $S_{2}:[0,1]^2\rightarrow[0,1]$ be two non-constant proper and strict t-subnorms with continuous, strictly decreasing additive generators. Then there exists a non-constant proper and strict t-subnorm $S_{u}:[0,1]^2\rightarrow[0,1]$ generated by a continuous, strictly decreasing additive generator such that
$$S_{u}\geq \max(S_{1}, S_{2}).$$
\end{theorem}
\begin{proof}Let $s_{n_{1}}, s_{n_{2}}:[0,1]\rightarrow[1,\infty]$ be the normalized additive generators of $S_{1}$ and $S_{2}$, respectively. Meanwhile, let $s:[0,1]\rightarrow[1,\infty]$ be a strictly decreasing bijection.

Consider the strictly decreasing bijection $f:[0,1]\rightarrow[1,\infty]$ given by $f(x)=\frac{1}{x}$ (taking into account $\frac{1}{\infty}=0$ and $\frac{1}{0}=\infty$). Then, the superadditivity of composite function $s\circ s_{n_{i}}^{-1}$ is equivalent with the superadditivity of composite function $m\circ m_{i}$, $i=1,2$, where $m=s\circ f^{-1}$ and $m_{i}=f\circ s_{n_{i}}^{-1}$, $i=1,2$, are strictly increasing bijections from $[1,\infty]$ to $[1,\infty]$. Therefore, it suffices to show that for two strictly increasing bijections $m_{1},m_{2}:[1,\infty]\rightarrow[1,\infty]$ there exists a strictly increasing bijection $m:[1,\infty]\rightarrow[1,\infty]$ such that both $m\circ m_{1}$ and $m\circ m_{2}$ are superadditive from Lemma \ref{lemm4.2}.

For two elements $1\leq x\leq y\leq \infty$, let
\begin{equation}\label{eq:6}
D(x,y)=
\begin{cases}
1 & \hbox{if }\ x=y,\\
\mbox{inf}\{h(x_{0}, x_{1}, \ldots, x_{n-1}, x_{n})| x=x_{0}<x_{1}<\cdots <x_{n}=y \mbox{ with }n\in\mathbb{N}\} & \hbox{if }\ {x<y},
\end{cases}
\end{equation}
where $x_{0}<x_{1}<\cdots <x_{n}$ is a finite sequence and
\begin{equation}\label{eq:7}
h(x_{0}, x_{1}, \ldots, x_{n-1}, x_{n})=\prod_{i=1}^{n}\mbox{min}\left(\frac{m_{1}^{-1}(x_{i-1})}{m_{1}^{-1}(x_{i})},\frac{m_{2}^{-1}(x_{i-1})}{m_{2}^{-1}(x_{i})}\right).
\end{equation}
It is evident that for all $i=1,2,\ldots,n$ we have $0<\frac{m_{1}^{-1}(x_{i-1})}{m_{1}^{-1}(x_{i})}<1$ and $0<\frac{m_{2}^{-1}(x_{i-1})}{m_{2}^{-1}(x_{i})}<1$ since both $m_{1}^{-1}$ and $m_{2}^{-1}$ are strictly increasing bijections from $[1,\infty]$ to $[1,\infty]$.

Further, let us define a function $m:[1,\infty]\rightarrow[1,\infty]$ by
\begin{equation}\label{eq:8}
m(x)=\frac{1}{D(1,x)}.
\end{equation}
In what follows, we shall prove that $m$ is the function we are seeking in four steps.\\
\newcounter{step}
\setcounter{step}{1} (\Roman{step}) Firstly, we claim that for all $x<y$,
\begin{equation}\label{eq:9}
\frac{m_{1}^{-1}(x)\cdot m_{2}^{-1}(x)}{m_{1}^{-1}(y)\cdot m_{2}^{-1}(y)}\leq D(x,y) \leq \mbox{min}\left(\frac{m_{1}^{-1}(x)}{m_{1}^{-1}(y)}, \frac{m_{2}^{-1}(x)}{m_{2}^{-1}(y)}\right).
\end{equation}
Indeed, since $0<\frac{m_{1}^{-1}(x_{i-1})}{m_{1}^{-1}(x_{i})}<1$ and $0<\frac{m_{2}^{-1}(x_{i-1})}{m_{2}^{-1}(x_{i})}<1$, we have
$$h(x_{i-1}, x_{i})=\mbox{min}\left(\frac{m_{1}^{-1}(x_{i-1})}{m_{1}^{-1}(x_{i})},\frac{m_{2}^{-1}(x_{i-1})}{m_{2}^{-1}(x_{i})}\right)>\frac{m_{1}^{-1}(x_{i-1})\cdot m_{2}^{-1}(x_{i-1})}{m_{1}^{-1}(x_{i})\cdot m_{2}^{-1}(x_{i})},$$
which implies
$$h(x_{0}, \ldots, x_{n})>\frac{m_{1}^{-1}(x_{0})\cdot m_{2}^{-1}(x_{0})}{m_{1}^{-1}(x_{n})\cdot m_{2}^{-1}(x_{n})}.$$
Hence, by Eq.(\ref{eq:6}) we get $D(x,y)\geq\frac{m_{1}^{-1}(x)\cdot m_{2}^{-1}(x)}{m_{1}^{-1}(y)\cdot m_{2}^{-1}(y)}$ for all $x<y$. Meanwhile, it is obvious that $D(x,y)\leq h(x,y)=\mbox{min}\left(\frac{m_{1}^{-1}(x)}{m_{1}^{-1}(y)}, \frac{m_{2}^{-1}(x)}{m_{2}^{-1}(y)}\right).$ \\
\stepcounter{step} (\Roman{step}) Secondly, it is evident that $D(x,y)\leq D(x,z)\cdot D(z,y)$ for any $z\in(x,y)$. On the other hand, for any sequence $x=x_{0}<x_{1}<\cdots <x_{n}=y$, we distinguish two cases as below.
\begin{itemize}
  \item If $z=x_{i}$ for some $i\in\{1,2,\ldots,n-1\}$, then
  $$h(x_{0},\ldots,x_{n})=h(x_{0},\ldots,x_{i})\cdot h(x_{i},\ldots,x_{n})=h(x_{0},\ldots,z)\cdot h(z,\ldots,x_{n}).$$
  \item If $z\in(x_{i},x_{i+1})$ for some $i\in\{0,1,\ldots,n-1\}$, then from Eq.(\ref{eq:7}) we have
  $$h(x_{i},z,x_{i+1})=h(x_{i},z)\cdot h(z,x_{i+1})\leq \frac{m_{1}^{-1}(x_{i})}{m_{1}^{-1}(z)}\cdot \frac{m_{1}^{-1}(z)}{m_{1}^{-1}(x_{i+1})}=\frac{m_{1}^{-1}(x_{i})}{m_{1}^{-1}(x_{i+1})},$$
  and similarly $h(x_{i},z,x_{i+1})\leq\frac{m_{2}^{-1}(x_{i})}{m_{2}^{-1}(x_{i+1})}$, thus
  \begin{align*} h(x_{0},\ldots,x_{n})&=h(x_{0},\ldots,x_{i})\cdot h(x_{i},x_{i+1})\cdot h(x_{i+1},\ldots,x_{n})\\ &\geq h(x_{0},\ldots,x_{i})\cdot h(x_{i},z,x_{i+1})\cdot h(x_{i+1},\ldots,x_{n})\\ &= h(x_{0},\ldots,x_{i})\cdot h(x_{i},z)\cdot h(z,x_{i+1})\cdot h(x_{i+1},\ldots,x_{n})\\ &=h(x_{0},\ldots,z)\cdot h(z,\ldots,x_{n}).\end{align*}
\end{itemize}
Therefore, $D(x,y)\geq D(x,z)\cdot D(z,y)$ for any $z\in(x,y)$. In summary,
\begin{equation}\label{eq:10}
D(x,y)=D(x,z)\cdot D(z,y)
\end{equation}
for any $1\leq x<z<y\leq \infty$. \\
\stepcounter{step} (\Roman{step}) Furthermore, we assert that $m:[1,\infty]\rightarrow[1,\infty]$ is a strictly increasing bijection. In fact, $D(x,y)<1$ for any $x<y$ since both $m_{1}^{-1}$ and $m_{2}^{-1}$ are strictly increasing bijections, and from Eq.(\ref{eq:10}) we have
$$D(1,y)=D(1,x)\cdot D(x,y),$$
i.e.,
\begin{equation}\label{eq:11}\frac{1}{m(y)}=\frac{1}{m(x)}\cdot D(x,y)
\end{equation}
 because of Eq.(\ref{eq:8}), hence $m(y)>m(x)$ whenever $x<y$. Therefore, $m$ is strictly increasing.

Additionally, the continuity of $m_{1}^{-1}$ and $m_{2}^{-1}$ yields that for $1\leq x<y\leq \infty$,
$$\lim_{x\rightarrow y^{-}}D(x,y)=\lim_{x\rightarrow y^{-}}h(x,y)=\lim_{y\rightarrow x^{+}}D(x,y)=\lim_{y\rightarrow x^{+}}h(x,y)=1.$$
Further, Eqs.\eqref{eq:9} and \eqref{eq:11} imply that
$$m(y)\cdot\frac{m_{1}^{-1}(x)\cdot m_{2}^{-1}(x)}{m_{1}^{-1}(y)\cdot m_{2}^{-1}(y)}\leq m(x) \leq h(x,y)\cdot m(y),$$
and consequently for all $1<y\leq\infty$,
$$m(y)=\lim_{x\rightarrow y^{-}}m(y)\cdot \frac{m_{1}^{-1}(x)\cdot m_{2}^{-1}(x)}{m_{1}^{-1}(y)\cdot m_{2}^{-1}(y)}\leq \lim_{x\rightarrow y^{-}}m(x)\leq \lim_{x\rightarrow y^{-}}h(x,y)\cdot m(y)=m(y),$$
i.e., $\displaystyle\lim_{x\rightarrow y^{-}}m(x)=m(y)$. Similarly, we have $\displaystyle\lim_{y\rightarrow x^{+}}m(y)=m(x)$ for all $1<x<\infty$.

Note that $m(1)=1$ and $m(\infty)=\infty$ since $D(1,\infty)=0$ and $\frac{1}{0}=\infty$. Therefore, $m$ is continuous on $[1,\infty]$.

In summary, $m:[1,\infty]\rightarrow[1,\infty]$ is a strictly increasing bijection.\\
\stepcounter{step} (\Roman{step}) Finally, we shall show that both $m\circ m_{1}$ and $m\circ m_{2}$ are superadditive.

From Eq.(\ref{eq:11}) we have for all $x<y$
\begin{equation*}
m(x)=D(x,y)\cdot m(y).
\end{equation*}
On the other hand, by Eq.\eqref{eq:9}, we have $$D(x,y)\leq \mbox{min}\left(\frac{m_{1}^{-1}(x)}{m_{1}^{-1}(y)}, \frac{m_{2}^{-1}(x)}{m_{2}^{-1}(y)}\right)\leq \frac{m_{1}^{-1}(x)}{m_{1}^{-1}(y)}.$$ Thus for all $x,y\in[1,\infty]$,
\begin{align*}
m(m_{1}(x))&=D(m_{1}(x), m_{1}(x+y))\cdot m(m_{1}(x+y))\\
&\leq \frac{m_{1}^{-1}(m_{1}(x))}{m_{1}^{-1}(m_{1}(x+y))}\cdot m(m_{1}(x+y))\\
&=\frac{x}{x+y}m(m_{1}(x+y)),
\end{align*}
and
$$m(m_{1}(y))\leq \frac{y}{x+y}m(m_{1}(x+y)),$$
yielding $m\circ m_{1}(x+y)\geq m\circ m_{1}(x)+m\circ m_{1}(y)$, i.e., $m\circ m_{1}$ is superadditive. Similarly we can show the superadditivity of $m\circ m_{2}$ from $D(x,y)\leq \mbox{min}\left(\frac{m_{1}^{-1}(x)}{m_{1}^{-1}(y)}, \frac{m_{2}^{-1}(x)}{m_{2}^{-1}(y)}\right)\leq \frac{m_{2}^{-1}(x)}{m_{2}^{-1}(y)}$ for all $x,y\in[1,\infty]$.
\end{proof}

By induction, the following corollary immediately follows from Theorem \ref{theorem:4.1}.
\begin{corollary}\label{coro4.2}
Let $S_{i}:[0,1]^2\rightarrow[0,1]$, $i=1,2,\ldots,n$, be a finite family of non-constant proper and strict t-subnorms with continuous, strictly decreasing additive generators. Then there exists a non-constant proper and strict t-subnorm $S_{u}:[0,1]^2\rightarrow[0,1]$ generated by a continuous, strictly decreasing additive generator such that
$$S_{u}\geq \max(S_{1}, S_{2},\ldots,S_{n}).$$
\end{corollary}

By duality, we have a completely analogous result about the lower bounds of t-superconorms generated by continuous, strictly increasing additive generators as the following corollary.
\begin{corollary}\label{coro4.3}
Let $M_{i}:[0,1]^2\rightarrow[0,1]$, $i=1,2,\ldots,n$, be a finite family of non-constant proper and strict t-superconorms with continuous, strictly increasing additive generators. Then there exists a non-constant proper and strict t-superconorm $M_{l}:[0,1]^2\rightarrow[0,1]$ generated by a continuous, strictly increasing additive generator such that
$$M_{l}\leq \min(M_{1}, M_{2},\ldots,M_{n}).$$
\end{corollary}

Let $S_{1}:[0,1]^2\rightarrow[0,1]$ and $S_{2}:[0,1]^2\rightarrow[0,1]$ be two non-constant proper and strict t-subnorms with continuous, strictly decreasing additive generators. Then from the proof of Theorem \ref{theorem:4.1}, we can summarize the following algorithm for obtaining an upper bound of both $S_{1}$ and $S_{2}$.

\begin{algorithm}[H]
\caption{\textbf{Finding an upper bound for $S_{1}$ and $S_{2}$}}
\label{alg:4.1}
\noindent
\begin{tabular}{p{0.15\linewidth}p{0.75\linewidth}}
\toprule
\textbf{Input} & Two non-constant proper and strict t-subnorms $S_{1}, S_{2}$ with their respective continuous, strictly decreasing additive generators $s_1$ and $s_2$ \\
\midrule
\textbf{Output} & A non-constant proper and strict t-subnorm $S_{u}$ with a continuous, strictly decreasing additive generator $s_n$\\
\midrule
\textbf{Step 1:} & Compute the normalized additive generators $s_{n_{1}}$ and $s_{n_{2}}$ of $S_{1}$ and $S_{2}$ by Eq.\eqref{eq5}, respectively. \\
\textbf{Step 2:} & Compute $m_{i}:[1,\infty]\rightarrow[1,\infty]$ defined by $m_{i}(x)=(f\circ s_{n_{i}}^{-1})(x)$, $i=1,2$, where $f:[0,1]\rightarrow[1,\infty]$ given by $f(x)=\frac{1}{x}$. \\
\textbf{Step 3:} & Compute $m_{i}^{-1}$, $i=1,2.$ \\
\textbf{Step 4:} & For all $x,y\in[1,\infty]$ with $x<y$, compute $D(x,y)$ defined by Eq.(\ref{eq:6}). \\
\textbf{Step 5:} & Compute $m:[1,\infty]\rightarrow[1,\infty]$ defined by Eq.(\ref{eq:8}). \\
\textbf{Step 6:} & Compute the normalized additive generators $s_{n}:[0,1]\rightarrow[1,\infty]$ by $s_{n}(x)=(m\circ f)(x)$. \\
\textbf{Step 7:} & Compute $S_{u}:[0,1]^2\rightarrow[0,1]$ by $S_{u}(x,y)=s_{n}^{(-1)}(s_{n}(x) + s_{n}(y))$. \\
\textbf{Step 8:} & End. \\
\bottomrule
\end{tabular}
\end{algorithm}

 Although Algorithm \ref{alg:4.1} isn't efficient because of Step 4, from the practical point of view, Step 4 of Algorithm 4.1 can be implemented through a numerical iteration method, whose core idea is to use an equidistant partition of the interval $[1, \infty]$ and repeatedly double the number of subintervals $[x_{i},x_{i+1}]$ with $i\in\{1,2,\ldots,n-1\}$, monotonically approximating the infimum until the absolute difference between two consecutive computational results falls below a given precision threshold. Even if the number of iterations required by the above method may grow logarithmically with the required precision, it is completely feasible for commonly used engineering-grade precision ($\varepsilon = 10^{-6}$ to $10^{-8}$). In particular, the following remark tells us the case that we can speed Algorithm \ref{alg:4.1}.
\begin{remark}\label{remak4.3}
In the proof of Theorem \ref{theorem:4.1}, if both $m_{1}^{-1}$ and $m_{2}^{-1}$ are further piecewise differentiable and
$$\frac{(m_{1}^{-1})'(x)}{m_{1}^{-1}(x)}\geq \frac{(m_{2}^{-1})'(x)}{m_{2}^{-1}(x)}$$
for all $x\in(x_{i},x_{i+1})$, then
$$D(x,y)=\frac{m_{1}^{-1}(x)}{m_{1}^{-1}(y)}$$
for all $x,y\in[x_{i},x_{i+1}]$ with $x<y$.
\end{remark}
\begin{proof}
Let $f(x)=\ln m_{1}^{-1}(x)$ and $g(x)=\ln m_{2}^{-1}(x)$. Then $\frac{(m_{1}^{-1})'(x)}{m_{1}^{-1}(x)}\geq \frac{(m_{2}^{-1})'(x)}{m_{2}^{-1}(x)}$ is equivalent to $f'(x)\geq g'(x)$ for all $x\in(x_{i},x_{i+1})$, and consequently
$$f(y)-f(x)=\int_{x}^{y}f'(t)\mbox{d}t\geq \int_{x}^{y}g'(t)\mbox{d}t=g(y)-g(x)$$
for all $x,y\in[x_{i},x_{i+1}]$ with $x<y$. Further,
$$\ln m_{1}^{-1}(y)-\ln m_{1}^{-1}(x)\geq \ln m_{2}^{-1}(y)-\ln m_{2}^{-1}(x),$$
i.e., $$\ln\frac{m_{1}^{-1}(x)}{m_{1}^{-1}(y)}\leq \ln\frac{m_{2}^{-1}(x)}{m_{2}^{-1}(y)},$$
or equivalently,
$$\frac{m_{1}^{-1}(x)}{m_{1}^{-1}(y)}\leq \frac{m_{2}^{-1}(x)}{m_{2}^{-1}(y)}.$$
Therefore, by Eqs.\eqref{eq:6} and \eqref{eq:7}, a simple calculation results in
$$D(x,y)=\frac{m_{1}^{-1}(x)}{m_{1}^{-1}(y)}.$$
\end{proof}

\begin{example}\label{exp4.2}\emph{Let $S_{1}:[0,1]^2\rightarrow[0,1]$ with $S_{1}(x,y)=\frac{1}{2}xy$ and $S_{2}:[0,1]^2\rightarrow[0,1]$ with $$S_{2}(x,y)=\frac{2}{1+\left[(\frac{2-x}{x})^{0.6}+(\frac{2-y}{y})^{0.6}\right]^{\frac{1}{0.6}}},$$
generated by their respective continuous, strictly decreasing additive generators $s_{1}:[0,1]\rightarrow[0,\infty]$ with $s_{1}(x)=-\ln \frac{x}{2}$ and $s_{2}:[0,1]\rightarrow[0,\infty]$ with $s_{2}(x)=(\frac{2}{x}-1)^{0.6}$. Obviously, both $S_{1}$ and $S_{2}$ are proper strict t-subnorms and they are incomparable (see Fig. \ref{fig2}). \\
Step1: The normalized additive generators are $s_{n_{1}}:[0,1]\rightarrow[1, \infty]$ with $s_{n_{1}}(x)=-\frac{\ln x}{\ln2}+1$ and $s_{n_{2}}:[0,1]\rightarrow[1, \infty]$ with $s_{n_{2}}(x)=(\frac{2}{x}-1)^{0.6}$, respectively. \\
Step2: $m_{1}(x)=f\circ s_{n_{1}}^{-1}(x)=\frac{1}{2^{1-x}}$ and $m_{2}(x)=f\circ s_{n_{2}}^{-1}(x)=\frac{x^{\frac{5}{3}}+1}{2}$. \\
Step3: $m_{1}^{-1}(x)=1+\log_{2}x$ and $m_{2}^{-1}(x)=(2x-1)^{\frac{3}{5}}$. \\
Step4: Obviously, $m_{1}^{-1}$ and $m_{2}^{-1}$ are differentiable, and $\frac{(m_{1}^{-1})'(x)}{m_{1}^{-1}(x)}=\frac{1}{x\ln(2x)}$, $\frac{(m_{2}^{-1})'(x)}{m_{2}^{-1}(x)}=\frac{6}{5(2x-1)}$. The transcendental equation $\frac{1}{x\ln(2x)}=\frac{6}{5(2x-1)}$ has a unique solution, denoted as $x_{0}$ ($x_{0}\approx 1.544$).
\begin{itemize}
  \item When $x\in[1,x_{0})$, we have $\frac{(m_{1}^{-1})'(x)}{m_{1}^{-1}(x)}\geq \frac{(m_{2}^{-1})'(x)}{m_{2}^{-1}(x)}$, then $D(x,y)=\frac{m_{1}^{-1}(x)}{m_{1}^{-1}(y)}=\frac{1+\log_{2}x}{1+\log_{2}y}$ for $1\leq x<y\leq x_{0}$ by Remark \ref{remak4.3}.
  \item When $x\in[x_{0},\infty]$, we have $\frac{(m_{1}^{-1})'(x)}{m_{1}^{-1}(x)}\leq \frac{(m_{2}^{-1})'(x)}{m_{2}^{-1}(x)}$, then $D(x,y)=\frac{m_{2}^{-1}(x)}{m_{2}^{-1}(y)}=\left(\frac{2x-1}{2y-1}\right)^{\frac{3}{5}}$ for $x_{0}\leq x<y\leq \infty$ by Remark \ref{remak4.3}.
\end{itemize}
Step5: From Eqs.\eqref{eq:8} and \eqref{eq:10} we have
\begin{itemize}
  \item $m(x)=\frac{1}{D(1,x)}=1+\log_{2}x$ whenever $x\in[1,x_{0})$.
  \item $m(x)=\frac{1}{D(1,x)}=\frac{1}{D(1,x_{0})\cdot D(x_{0},x)}=(1+\log_{2}x_{0})\cdot\frac{(2x-1)^{\frac{3}{5}}}{(2x_{0}-1)^{\frac{3}{5}}}$ since $D(1,x_{0})=\frac{1}{1+\log_{2}x_{0}}$ and $D(x_{0},x)=\left(\frac{2x_{0}-1}{2x-1}\right)^{\frac{3}{5}}$ whenever $x\in[x_{0},\infty]$.
\end{itemize}
Step6: Because $s_{n}(x)=m\circ f(x)$ and $f(x)=\frac{1}{x}$, we get
\begin{equation*}
s_{n}(x)=
\begin{cases}
(1+\log_{2}x_{0})\cdot\frac{(\frac{2}{x}-1)^{\frac{3}{5}}}{(2x_{0}-1)^{\frac{3}{5}}} & \hbox{if }\ x\in[0,\frac{1}{x_{0}}],\\
1-\log_{2}x & \hbox{if }\ x\in[\frac{1}{x_{0}},1].
\end{cases}
\end{equation*} It is evident that $s_{n}$ is a strictly decreasing bijection from $[0,1]$ to $[1,\infty]$ (see Fig. \ref{fig3}).\\
Step7: Using Eq.(\ref{eq:1}) and substituting $s_{n}$, we can construct a proper and strict t-subnorm $S_{u}$. By Theorem \ref{theorem:4.1}, $S_{u}\geq \max(S_{1}, S_{2})$. \\
Step8: End.}
\end{example}

\begin{figure}[H]
\centering
\includegraphics[width=0.7\textwidth, keepaspectratio]{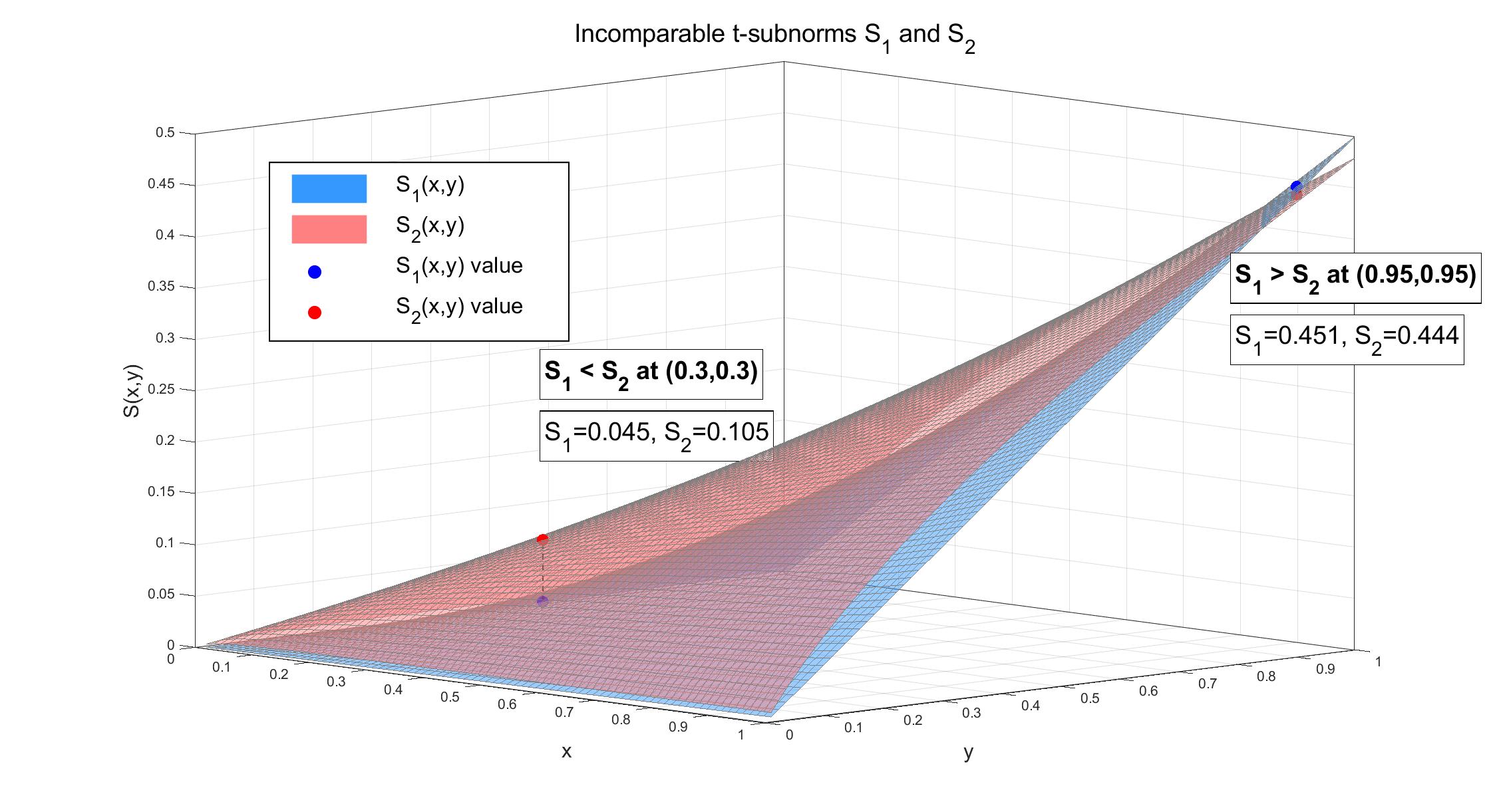}
\caption{3D plot of two proper and strict t-subnorms in Example \ref{exp4.2}.}
\label{fig2}
\end{figure}

\begin{figure}[H]
\centering
\includegraphics[width=0.6\textwidth, keepaspectratio]{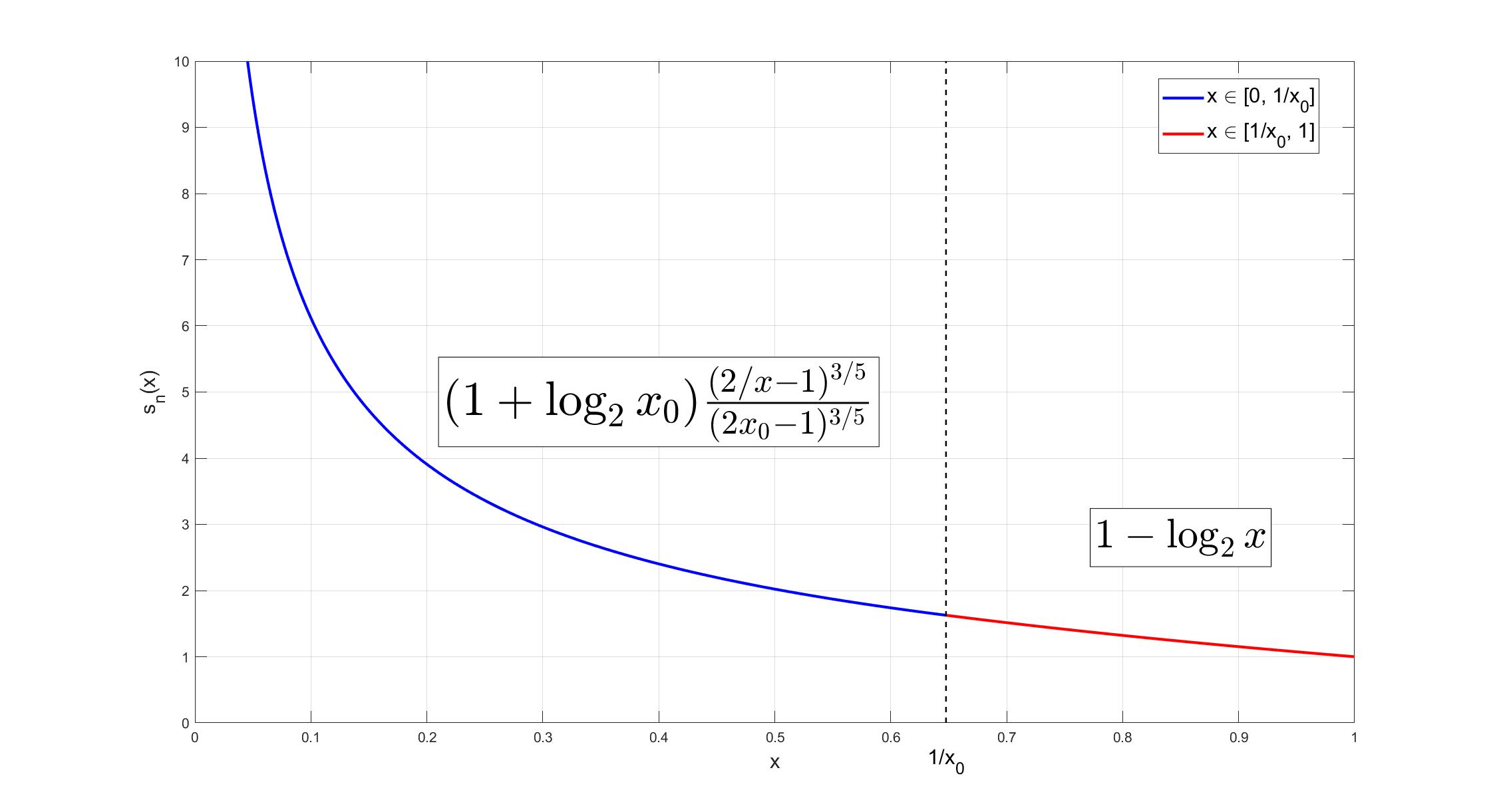}
\caption{Additive generator $s_{n}$ of $S_{u}$ in Example \ref{exp4.2}.}
\label{fig3}
\end{figure}

In summary, from Theorem \ref{thm4.1}, Proposition \ref{pro4.3}, Corollary \ref{coro4.1} and Theorem \ref{theorem:4.1} we immediately get the following consequence.
\begin{theorem}\label{theorem:4.2}
Let $S_{1}:[0,1]^2\rightarrow[0,1]$ and $S_{2}:[0,1]^2\rightarrow[0,1]$ be two t-subnorms with continuous, strictly decreasing additive generators. If both $S_{1}$ and $S_{2}$ are strict, then there exists a strict t-subnorm $S_{u}:[0,1]^2\rightarrow[0,1]$ generated by a continuous, strictly decreasing additive generator such that
$$S_{u}\geq \max(S_{1}, S_{2}).$$
\end{theorem}

\subsection{In the case of proper and nilpotent t-subnorms}
%In the case that both $S_{1}$ and $S_{2}$ are proper and nilpotent, we partially demonstrate the existence of an upper bound and, moreover, provide an algorithm for its construction.
Analogous to the proof of Theorem \ref{theorem:3.1}, we have the following lemma.
\begin{lemma}\label{lemm4.3}
Let $S_{1}:[0,1]^2\rightarrow[0,1]$ and $S_{2}:[0,1]^2\rightarrow[0,1]$ be two non-constant proper and nilpotent t-subnorms with their respective normalized additive generators $s_{n_{1}}:[0,1]\rightarrow[1, s_{n_{1}}(0)]$ and $s_{n_{2}}:[0,1]\rightarrow[1, s_{n_{2}}(0)]$. Then $S_{1}\leq S_{2}$ if and only if the function $s_{_{n_{2}}}\circ s_{n_{1}}^{-1}:[1,s_{n_{1}}(0)]\rightarrow[1, s_{n_{2}}(0)]$ is superadditive, i.e., for all $u,v\in[1, s_{n_{1}}(0)]$ with $u+v\in[1, s_{n_{1}}(0)]$ we have
\begin{equation*}
s_{n_{2}}\circ s_{n_{1}}^{-1}(u+v)\geq s_{n_{2}}\circ s_{n_{1}}^{-1}(u)+s_{n_{2}}\circ s_{n_{1}}^{-1}(v).
\end{equation*}
\end{lemma}

Based on Lemma \ref{lemm4.3}, we can establish the following theorem.

\begin{theorem}\label{theorem:4.3}
Let $S_{1}:[0,1]^2\rightarrow[0,1]$ and $S_{2}:[0,1]^2\rightarrow[0,1]$ be two non-constant proper and nilpotent t-subnorms with their respective normalized additive generators $s_{_{n_{1}}}$ and $s_{_{n_{2}}}$. If $s_{n_{1}}(0)=s_{_{n_{2}}}(0)$, then there exists a non-constant proper and nilpotent t-subnorm $S_{u}:[0,1]^2\rightarrow[0,1]$ generated by a continuous, strictly decreasing additive generator such that
$$S_{u}\geq \max(S_{1}, S_{2}).$$
\end{theorem}
\begin{proof}Because $S_{1}$ and $S_{2}$ are proper and nilpotent, by Remark \ref{remark:4.1}(i) we have $1<s_{n_{1}}(0)=s_{_{n_{2}}}(0)< \infty$. Without loss of generality we assume that $s_{n_{1}}(0)=s_{_{n_{2}}}(0)=c$ with $c\in(1,\infty)$, namely, both $s_{_{n_{1}}}$ and $s_{_{n_{2}}}$ are continuous, strictly decreasing bijections from $[0,1]$ to $[1,c]$. Define a strictly decreasing bijection $f:[0,1]\rightarrow[1,c]$ by
$$f(x)=(1-c)x+c$$
and set $\xi_{i}=f\circ s_{n_{i}}^{-1}$, $i=1,2$. Then each $\xi_{i}$ is a strictly increasing bijection from $[1,c]$ to $[1,c]$. Denote its inverse by $\psi_{i}=\xi_{i}^{-1}=s_{n_{i}}\circ f^{-1}$. In the next, we shall construct a strictly increasing bijection $\xi:[1,c]\rightarrow[1,d]$ with $d>1$ such that both $\xi\circ \xi_{1}$ and $\xi\circ \xi_{2}$ are superadditive.

Indeed, for any finite sequence $1\leq x_{0}<x_{1}<\cdots <x_{n}\leq c$ define
\begin{equation*}
h(x_{0}, x_{1}, \ldots, x_{n-1}, x_{n})=\prod_{i=1}^{n}\mbox{min}\left(\frac{\psi_{1}(x_{i-1})}{\psi_{1}(x_{i})},\frac{\psi_{2}(x_{i-1})}{\psi_{2}(x_{i})}\right),
\end{equation*}
and for any $1\leq x\leq y\leq c$ let
\begin{equation}\label{eq:12}
D(x,y)=
\begin{cases}
1 & \hbox{if }\ x=y,\\
\mbox{inf}\{h(x_{0}, x_{1}, \ldots, x_{n-1}, x_{n})| x=x_{0}<x_{1}<\cdots <x_{n}=y \mbox{ with }n\in\mathbb{N}\} & \hbox{if }\ {x<y}.
\end{cases}
\end{equation}
Similar to Theorem \ref{theorem:4.1}, the function $D$ enjoys the following properties:
\begin{enumerate}
\item[(i)] For any $z\in(x,y)$, $D(x,y)=D(x,z)\cdot D(z,y)$.
\item[(ii)] For any $x<y$,
\begin{equation*}
\frac{\psi_{1}(x)\cdot\psi_{2}(x)}{\psi_{1}(y)\cdot\psi_{2}(y)}\leq D(x,y) \leq \mbox{min}\left(\frac{\psi_{1}(x)}{\psi_{1}(y)}, \frac{\psi_{2}(x)}{\psi_{2}(y)}\right).
\end{equation*}
  In particular, $D(x,y)<1$ whenever $x<y$.
\end{enumerate}
Now define a function $\xi:[1,c]\rightarrow[1,\infty)$ by
\begin{equation}\label{eq:13}
\xi(x)=\frac{1}{D(1,x)}.
\end{equation}
From (i) we obtain
\begin{equation*}
\xi(x)=\xi(y)\cdot D(x,y)
\end{equation*}
 for any $x<y$. Then $\xi(1)=1$ and, the function $\xi$ is strictly increasing since $D(x,y)<1$ whenever $x<y$. Set $d=\xi(c)=\frac{1}{D(1,c)}$. Thus $1< d <\infty$. Hence $\xi$ is a strictly increasing bijection from $[1,c]$ to $[1,d]$. The proof of the superadditivity of $\xi\circ \xi_{i}$, $i=1,2$, is entirely analogous to that of Theorem \ref{theorem:4.1}.

Finally, define a mapping $s_{n}:[0,1]\rightarrow[1,d]$ by $s_{n}=\xi\circ f$. Because $\xi$ is strictly increasing and $f$ strictly decreasing, $s_{n}$ is strictly decreasing. The fact that both $\xi$ and $f$ are continuous indicates that $s_{n}$ is continuous. Moreover, $s_{n}(1)=\xi\circ f(1)=\xi(1)=1$ and $s_{n}(0)=\xi\circ f(0)=\xi(c)=d$. Hence $s_{n}$ is a continuous, strictly decreasing bijection and therefore a normalized additive generator of a proper nilpotent non-constant t-subnorm $S_{u}$, in which $1< d <\infty$ guarantees the nilpotency of $S_{u}$ by Remark \ref{remark:4.1}(i)). Meanwhile, from $s_{n}\circ s_{n_{i}}^{-1}=(\xi\circ f)\circ(f^{-1}\circ\xi_i)=\xi\circ\xi_{i}$ and \( \xi\circ\xi_i \) being superadditive, both $s_{n}\circ s_{n_{1}}^{-1}$ and $s_{n}\circ s_{n_{2}}^{-1}$ are superadditive. Consequently $S_{u}\ge S_{i}$ for $i=1,2$ from Lemma \ref{lemm4.3}, i.e., $S_{u}\ge\max(S_{1},S_{2})$.
\end{proof}

\begin{remark}\label{remark:4.3} \emph{ In Theorem \ref{theorem:4.3}, the condition $s_{n_{1}}(0)=s_{_{n_{2}}}(0)$ cannot be dropped, in general. For example, define the functions $s_{n_{1}}:[0,1]\rightarrow[1,2]$ with $s_{n_{1}}(x)=2-x$ and $s_{n_{2}}:[0,1]\rightarrow[1,3]$ with $s_{n_{2}}(x)=3-2x$. Assume that there exists a function $f:[0,1]\rightarrow[1,c]$ such that both $f\circ s_{n_{1}}^{-1}:[1,2]\rightarrow[1,c]$ and $f\circ s_{n_{2}}^{-1}:[1,3]\rightarrow[1,c]$ are strictly increasing bijections. These result in $c=2$ and $c=3$, respectively, contradiction.}
\end{remark}

Let $S_{1}:[0,1]^2\rightarrow[0,1]$ and $S_{2}:[0,1]^2\rightarrow[0,1]$ be two non-constant proper and nilpotent t-subnorms with continuous, strictly decreasing additive generators. Then from the proof of Theorem \ref{theorem:4.3}, we can summarize an algorithm to find an upper bound of both $S_{1}$ and $S_{2}$ as follows.

\begin{algorithm}[H]
\caption{\textbf{Finding an upper bound for $S_{1}$ and $S_{2}$}}
\label{alg:4.2}
\noindent
\begin{tabular}{p{0.15\linewidth}p{0.75\linewidth}}
\toprule
\textbf{Input} & Two non-constant proper and nilpotent t-subnorms $S_{1}, S_{2}$ with their respective continuous, strictly decreasing additive generators $s_1$ and $s_2$, satisfying $s_{n_{1}}(0)=s_{_{n_{2}}}(0)$ \\
\midrule
\textbf{Output} & A non-constant proper and nilpotent t-subnorm $S_{u}$ with a continuous, strictly decreasing additive generator $s_n$ \\
\midrule
\textbf{Step 1:} & Compute the normalized additive generators $s_{n_{1}}$ and $s_{n_{2}}$ of $S_{1}$ and $S_{2}$ by Eq.\eqref{eq5}, respectively. \\
\textbf{Step 2:} & Compute $\xi_{i}:[1,c]\rightarrow [1,c]$ defined by $\xi_{i}=f\circ s_{n_{i}}^{-1}$, $i=1,2$, where $f:[0,1]\rightarrow[1,c]$ is given by $f(x)=(1-c)x+c$. \\
\textbf{Step 3:} & Compute $\psi_{i}=\xi_{i}^{-1}$, $i=1,2.$ \\
\textbf{Step 4:} & For all $x,y\in[1,c]$ with $x<y$, compute $D(x,y)$ defined by Eq.(\ref{eq:12}). \\
\textbf{Step 5:} & Compute $\xi:[1,c]\rightarrow[1,d]$ defined by Eq.(\ref{eq:13}). \\
\textbf{Step 6:} & Compute the normalized additive generators $s_{n}:[0,1]\rightarrow[1,d]$ by $s_{n}(x)=(\xi\circ f)(x)$. \\
\textbf{Step 7:} & Compute $S_{u}:[0,1]^2\rightarrow[0,1]$ by $S_{u}(x,y)=s_{n}^{(-1)}(s_{n}(x) + s_{n}(y))$. \\
\textbf{Step 8:} & End. \\
\bottomrule
\end{tabular}
\end{algorithm}

Note that Algorithm \ref{alg:4.2} has a very large computational complexity since Step 4.
\begin{example}\label{exp4.3}\emph{Let $S_{1}:[0,1]^2\rightarrow[0,1]$ with
$$S_{1}(x,y)=\mbox{max}\left(2-\sqrt[3]{(2-x)^{3}+(2-y)^{3}},0\right)$$
and $S_{2}:[0,1]^2\rightarrow[0,1]$ with
$$S_{2}(x,y)=\mbox{max}\left(\frac{x+y-\frac{1}{2^{\frac{7}{8}}-1}+(2^{\frac{7}{8}}-1)xy}{2},0\right)$$
generated by their respective continuous, strictly decreasing additive generators $s_{1}:[0,1]\rightarrow[\frac{1}{8},1]$ with $s_{1}(x)=(1-\frac{x}{2})^{3}$ and $s_{2}:[0,1]\rightarrow[\frac{1}{8},1]$ with $s_{2}(x)=1-\frac{\ln(1+(2^{\frac{7}{8}}-1)x)}{\ln2}$. Obviously, both $S_{1}$ and $S_{2}$ are proper nilpotent t-subnorms and they are incomparable (see Fig. \ref{fig4}). \\[5pt]
Step1: The normalized additive generators are $s_{n_{1}}:[0,1]\rightarrow[1, 8]$ with $s_{n_{1}}(x)=(2-x)^{3}$ and $s_{n_{2}}:[0,1]\rightarrow[1, 8]$ with $s_{n_{2}}(x)=8(1-\frac{\ln(1+(2^{\frac{7}{8}}-1)x)}{\ln2})$, respectively. Thus, $c=s_{n_{1}}(0)=s_{_{n_{2}}}(0)=8$. \\[5pt]
Step2: $f(x)=-7x+8$, $\xi_{1}(x)=f\circ s_{n_{1}}^{-1}(x)=7\sqrt[3]{x}-6$ and $\xi_{2}(x)=f\circ s_{n_{2}}^{-1}(x)=-7\cdot \frac{2^{1-\frac{x}{8}}-1}{2^{\frac{7}{8}}-1}+8$. \\[5pt]
Step3: $\psi_{1}(x)=\xi_{1}^{-1}(x)=(\frac{6+x}{7})^{3}$ and $\psi_{2}(x)=\xi_{2}^{-1}(x)=8(1-\frac{\ln(1+\frac{\delta-1}{7}(8-x))}{\ln2})$, where $\delta = 2^{\frac{7}{8}}$. \\[5pt]
Step4: Obviously, $\psi_{1}$ and $\psi_{2}$ are differentiable, $\frac{\psi_{1}'(x)}{\psi_{1}(x)}=\frac{3}{6+x}$ and
$$\frac{\psi_2'(x)}{\psi_2(x)} = \frac{\delta - 1}{7} \cdot \frac{1}{\left[1 + \frac{\delta - 1}{7}(8 - x)\right] \ln\left( \frac{2}{1+\frac{\delta - 1}{7}(8-x)} \right)}.$$
The transcendental equation $\frac{\psi_{1}'(x)}{\psi_{1}(x)}=\frac{\psi_{2}'(x)}{\psi_{2}(x)}$ has a unique solution, denoted as $x_{0}$ ($x_{0}\approx 2.943$).
\begin{itemize}
  \item When $x\in[1,x_{0})$, we have $\frac{\psi_{1}'(x)}{\psi_{1}(x)}<\frac{\psi_{2}'(x)}{\psi_{2}(x)}$, then $D(x,y)=\frac{\psi_{2}(x)}{\psi_{2}(y)}=\frac{\ln2-\ln(1+\frac{\delta-1}{7}(8-x))}{\ln2-\ln(1+\frac{\delta-1}{7}(8-y))}$ for $1\leq x<y\leq x_{0}$.
  \item When $x\in[x_{0},8]$, we have $\frac{\psi_{1}'(x)}{\psi_{1}(x)}>\frac{\psi_{2}'(x)}{\psi_{2}(x)}$, then $D(x,y)=\frac{\psi_{1}(x)}{\psi_{1}(y)}=(\frac{6+x}{6+y})^{3}$ for $x_{0}\leq x<y\leq 8$.
\end{itemize}
Step5: Further, we have
\begin{itemize}
  \item $\xi(x)=\frac{1}{D(1,x)}=8(1-\frac{\ln(1+\frac{\delta-1}{7}(8-x))}{\ln2})$ whenever $x\in[1,x_{0})$.
  \item $\xi(x)=\frac{1}{D(1,x)}=\frac{1}{D(1,x_{0})\cdot D(x_{0},x)}=8(1-\frac{\ln(1+\frac{\delta-1}{7}(8-x_{0}))}{\ln2})\cdot (\frac{6+x}{6+x_{0}})^{3}$ since $D(1,x_{0})=\frac{\ln2}{8(\ln2-\ln(1+\frac{\delta-1}{7}(8-x_{0})))}$ and $D(x_{0},x)=(\frac{6+x_{0}}{6+x})^{3}$ whenever $x\in[x_{0},8]$.
\end{itemize}
Step6: By $s_{n}(x)=\xi\circ f(x)$ and $f(x)=-7x+8$, we get
\begin{equation*}
 s_{n}(x)=\begin{cases}
8(1-\frac{\ln(1+\frac{\delta-1}{7}(8-x_{0}))}{\ln2})\cdot (\frac{14-7x}{6+x_{0}})^{3} & \hbox{if }\ x\in[0,\frac{8-x_{0}}{7}],\\[5pt]
8(1-\frac{\ln(1+(\delta-1)x)}{\ln2}) & \hbox{if }\ x\in[\frac{8-x_{0}}{7},1].
\end{cases}
\end{equation*} It is evident that $s_{n}$ is a strictly decreasing bijection from $[0,1]$ to $[1,d]$, where $1<d=\frac{1}{D(1,8)}<\infty$ (see Fig. \ref{fig5}). \\[5pt]
Step7: Using Eq.(\ref{eq:1}) and substituting $s_{n}$, we can construct a proper and nilpotent t-subnorm $S_{u}$. By Theorem \ref{theorem:4.3}, $S_{u}\geq \max(S_{1}, S_{2})$. \\[5pt]
Step8: End.}
\end{example}

\begin{figure}[H]
\centering
\includegraphics[width=0.8\textwidth, keepaspectratio]{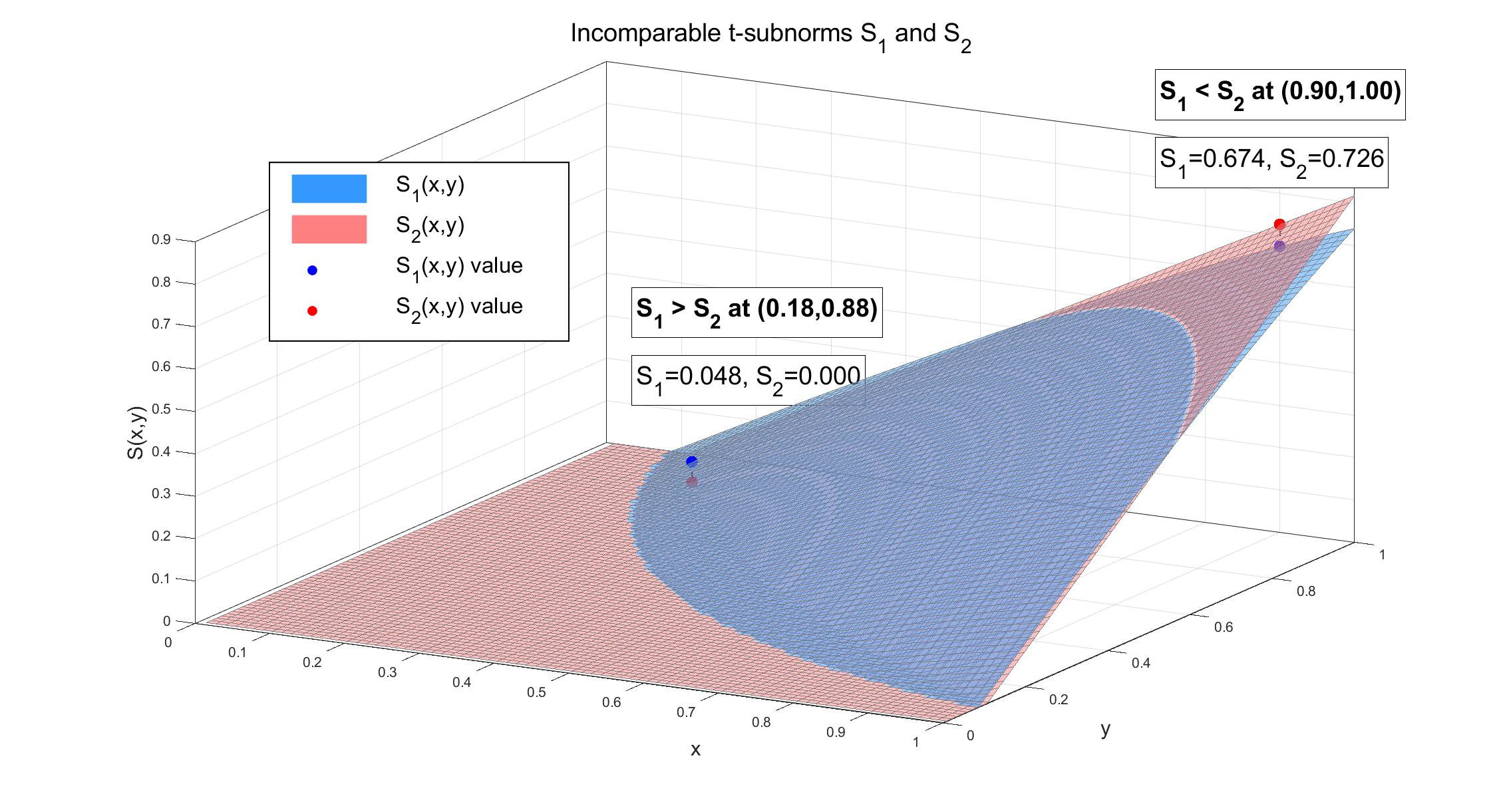}
\caption{3D plot of two proper and nilpotent t-subnorms in Example \ref{exp4.3}.}
\label{fig4}
\end{figure}

\begin{figure}[H]
\centering
\includegraphics[width=0.8\textwidth, keepaspectratio]{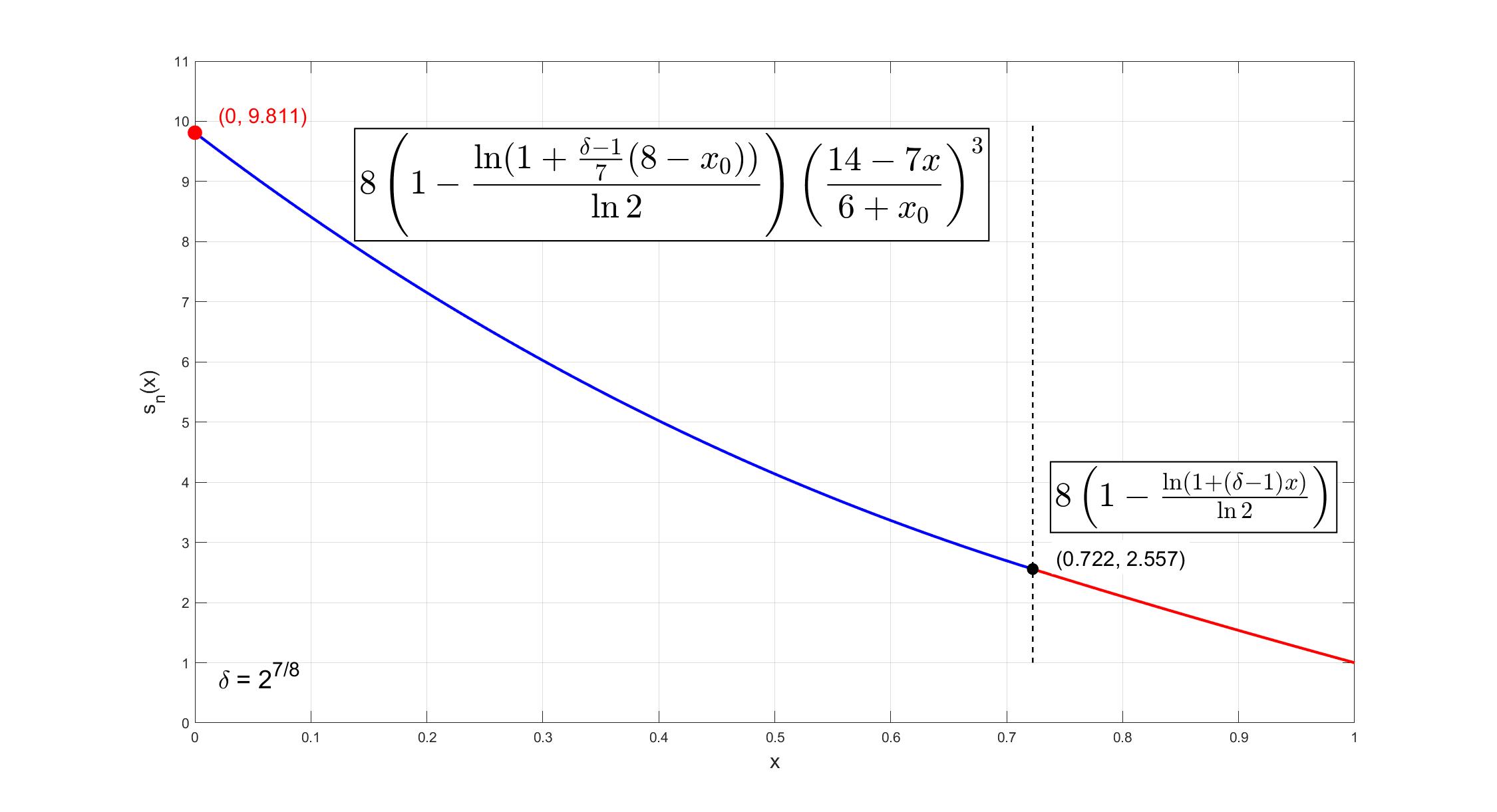}
\caption{Additive generator $s_{n}$ of $S_{u}$ in Example \ref{exp4.3}.}
\label{fig5}
\end{figure}

\section{Lower bounds of t-subnorms with continuous, strictly decreasing additive generators}
This section focuses on the existence of a lower bound of both t-subnorms $S_{1}$ and $S_{2}$  with continuous, strictly decreasing additive generators, and an algorithm is given for calculating it.

The following lemma is a basis of the coming discussion.
\begin{lemma}\label{lemm5.1}
Let $f:(0,\infty)\rightarrow(0,\infty)$. If the function $\frac{f(x)}{x}$ is non-increasing on $(0,\infty)$, then $f$ is subadditive.
\end{lemma}
\begin{proof}
Since $\frac{f(x)}{x}$ is non-increasing on $(0,\infty)$, for any $x,y\in(0,\infty)$ we have
$$\frac{f(x)}{x}\geq\frac{f(x+y)}{x+y} \quad\mbox{and}\quad \frac{f(y)}{y}\geq\frac{f(x+y)}{x+y},$$
or, equivalently,
$$f(x)\geq\frac{x}{x+y}f(x+y) \quad\mbox{and}\quad f(y)\geq\frac{y}{x+y}f(x+y).$$
Therefore, $f(x)+f(y)\geq f(x+y)$, i.e., $f$ is subadditive.
\end{proof}

\subsection{In the case of proper and strict t-subnorms}
%In the case that both $S_{1}$ and $S_{2}$ are proper and strict with continuous, strictly decreasing additive generators, we completely show the existence of a lower bound of both $S_{1}$ and $S_{2}$. In addition, we present an algorithm for finding a lower bound for $S_{1}$ and $S_{2}$.

\begin{theorem}\label{theorem:5.1}
Let $S_{1}:[0,1]^2\rightarrow[0,1]$ and $S_{2}:[0,1]^2\rightarrow[0,1]$ be two non-constant proper and strict t-subnorms with continuous, strictly decreasing additive generators. Then there exists a non-constant proper and strict t-subnorm $S_{l}:[0,1]^2\rightarrow[0,1]$ generated by a continuous, strictly decreasing additive generator such that
$$S_{l}\leq \min(S_{1}, S_{2}).$$
\end{theorem}
\begin{proof}Let $s_{n_{1}}, s_{n_{2}}:[0,1]\rightarrow[1,\infty]$ be the normalized additive generators of $S_{1}$ and $S_{2}$, respectively. From Theorem \ref{theorem:3.1}, it suffices to show that there exists a strictly decreasing bijection $s:[0,1]\rightarrow[1,\infty]$ such that both $s\circ s_{n_{1}}^{-1}$ and $s\circ s_{n_{2}}^{-1}$ are subadditive, which will be completed by splitting in three steps. \\
\setcounter{step}{1} (\Roman{step}) Firstly, we define function $s\circ s_{n_{2}}^{-1}:[1,\infty]\rightarrow[1,\infty]$ as follows.
\begin{itemize}
  \item If $x=2^{k}$, $k=0,1,2,3,\ldots$, define
  \begin{equation}\label{eq:14}
 s\circ s_{n_{2}}^{-1}(2^{k})=\begin{cases}
1 & \ k=0,\\
\mbox{min}\left(s\circ s_{n_{2}}^{-1}(2^{k-1})\cdot \frac{s_{n_{1}}\circ s_{n_{2}}^{-1}(2^{k})}{s_{n_{1}}\circ s_{n_{2}}^{-1}(2^{k-1})}, s\circ s_{n_{2}}^{-1}(2^{k-1})+\frac{1}{2\cdot s\circ s_{n_{2}}^{-1}(2^{k-1})}\right) & \mbox{otherwise}.
\end{cases}
\end{equation}

  \item If $x\in[2^{k},2^{k+1}]$, $k=0,1,2,3,\ldots$, define
  \begin{equation}\label{eq:15}
  s\circ s_{n_{2}}^{-1}(x)=A\cdot \left(s_{n_{1}}\circ s_{n_{2}}^{-1}\right)(x)+B,
  \end{equation}
where both $A=\frac{s\circ s_{n_{2}}^{-1}(2^{k+1})-s\circ s_{n_{2}}^{-1}(2^{k})}{s_{n_{1}}\circ s_{n_{2}}^{-1}(2^{k+1})-s_{n_{1}}\circ s_{n_{2}}^{-1}(2^{k})}$ and $B=s\circ s_{n_{2}}^{-1}(2^{k})-A\cdot \left(s_{n_{1}}\circ s_{n_{2}}^{-1}\right)(2^{k})$ are constant because of the well-defined $s\circ s_{n_{2}}^{-1}(2^{k})$, $k=0,1,2,3,\ldots$, as well as the known composite function $s_{n_{1}}\circ s_{n_{2}}^{-1}$.
\end{itemize}

Then we have the following three statements:\\
{\bf A}. The function $s\circ s_{n_{1}}^{-1}$ is linear on $[s_{n_{1}}\circ s_{n_{2}}^{-1}(2^{k}),s_{n_{1}}\circ s_{n_{2}}^{-1}(2^{k+1})]$ for $k=0,1,2,3,\ldots$.

Indeed, let $y=s_{n_{1}}\circ s_{n_{2}}^{-1}(x), x\in[2^{k},2^{k+1}]$. Then $x=s_{n_{2}}\circ s_{n_{1}}^{-1}(y), y\in [s_{n_{1}}\circ s_{n_{2}}^{-1}(2^{k}),s_{n_{1}}\circ s_{n_{2}}^{-1}(2^{k+1})]$. From Eq.(\ref{eq:15}) we get
$$s\circ s_{n_{1}}^{-1}(y)=\left(s\circ s_{n_{2}}^{-1}\right)\circ \left(s_{n_{2}}\circ s_{n_{1}}^{-1}\right)(y)=A\cdot \left(s_{n_{1}}\circ s_{n_{2}}^{-1}\right)\circ \left(s_{n_{2}}\circ s_{n_{1}}^{-1}\right)(y)=Ay+B,$$
showing that $s\circ s_{n_{1}}^{-1}$ is linear. \\
 {\bf B}. $s\circ s_{n_{2}}^{-1}$ is strictly increasing and continuous on $[1,\infty)$.

 It is clear that $A>0$ since $s\circ s_{n_{2}}^{-1}$ and $s_{n_{1}}\circ s_{n_{2}}^{-1}$ are monotonically increasing, which together with the continuity of $s_{n_{1}}\circ s_{n_{2}}^{-1}$ and Eq.(\ref{eq:15}) deduces that $s\circ s_{n_{2}}^{-1}$ is strictly increasing and continuous on $[1,\infty)$.\\
{\bf C}. $s\circ s_{n_{2}}^{-1}(\infty)=\infty$.

In fact, assuming that $\displaystyle\lim_{x\rightarrow\infty}s\circ s_{n_{2}}^{-1}(x)=L<\infty$ with $L>1$, we have
  \begin{equation}\label{eq:16}
  \lim_{k\rightarrow\infty}\left(s\circ s_{n_{2}}^{-1}(2^{k+1})-s\circ s_{n_{2}}^{-1}(2^{k})\right)=0.
  \end{equation}
From Eq.(\ref{eq:14}), if there are subsequences $(k_i)_{i=1,\cdots, n,\cdots}$ of $\{0,1,\cdots, n,\cdots\}$ such that
$$s\circ s_{n_{2}}^{-1}(2^{k_i})= s\circ s_{n_{2}}^{-1}(2^{k_i-1})+\frac{1}{2\cdot s\circ s_{n_{2}}^{-1}\left(2^{k_i-1}\right)},$$
i.e.,
$$s\circ s_{n_{2}}^{-1}(2^{k_i})-s\circ s_{n_{2}}^{-1}(2^{k_i-1})= \frac{1}{2\cdot s\circ s_{n_{2}}^{-1}\left(2^{k_i-1}\right)},$$
 then $\displaystyle\lim_{k_i\rightarrow\infty}\frac{1}{2\cdot s\circ s_{n_{2}}^{-1}\left(2^{k_i-1}\right)}=\frac{1}{2L}>0$, contrary to Eq.(\ref{eq:16}). Therefore, there exists $k_{0}\in \mathbb{N}$ such that for all $k>k_{0}$,
$$s\circ s_{n_{2}}^{-1}(2^{k+1})=s\circ s_{n_{2}}^{-1}(2^{k})\cdot \frac{s_{n_{1}}\circ s_{n_{2}}^{-1}(2^{k+1})}{s_{n_{1}}\circ s_{n_{2}}^{-1}(2^{k})},$$
i.e., $$s\circ s_{n_{2}}^{-1}(2^{k+1})=s_{n_{1}}\circ s_{n_{2}}^{-1}(2^{k+1})\cdot \frac{s\circ s_{n_{2}}^{-1}(2^{k})}{s_{n_{1}}\circ s_{n_{2}}^{-1}(2^{k})}.$$
Then, by induction, for any $n\in \mathbb{N}$,
$$s\circ s_{n_{2}}^{-1}(2^{k+n})=\left(s_{n_{1}}\circ s_{n_{2}}^{-1}\right)(2^{k+n})\cdot \frac{s\circ s_{n_{2}}^{-1}(2^{k})}{s_{n_{1}}\circ s_{n_{2}}^{-1}(2^{k})}.$$
Recall that $s_{n_{1}}\circ s_{n_{2}}^{-1}:[1,\infty]\rightarrow[1,\infty]$ is a strictly increasing bijection. Thus $\displaystyle\lim_{n\rightarrow\infty}s_{n_{1}}\circ s_{n_{2}}^{-1}(2^{k+n})=\infty$. So that $\displaystyle\lim_{n\rightarrow\infty}s\circ s_{n_{2}}^{-1}(2^{k+n})=\infty$, contradicting the assumption $\displaystyle\lim_{x\rightarrow\infty}s\circ s_{n_{2}}^{-1}(x)=L<\infty$.

Statements B and C imply that $s\circ s_{n_{2}}^{-1}$ is a strictly increasing bijection from $[1,\infty]$ to $[1,\infty]$, which is equivalent to $s:[0,1]\rightarrow[1,\infty]$ being a strictly decreasing bijection.\\
\stepcounter{step} (\Roman{step}) We claim that $s\circ s_{n_{2}}^{-1}$ is subadditive.

Indeed, for any $x,y\in[1,\infty]$, without loss of generality, assume $1\leq x\leq y<\infty$ and let $2^{k-1}\leq y< 2^{k}$ with $k\in\mathbb{N}$. Since $s\circ s_{n_{2}}^{-1}$ is strictly increasing, we have
\begin{align*}
s\circ s_{n_{2}}^{-1}(x+y)&\leq s\circ s_{n_{2}}^{-1}(y+y)\leq s\circ s_{n_{2}}^{-1}(2^{k}+2^{k})=s\circ s_{n_{2}}^{-1}(2^{k+1})\\
&\leq s\circ s_{n_{2}}^{-1}(2^{k})+\frac{1}{2} \mbox{ by Eq. }(\ref{eq:14})\\
&\leq s\circ s_{n_{2}}^{-1}(2^{k-1})+\frac{1}{2}+\frac{1}{2}\\
&=s\circ s_{n_{2}}^{-1}(2^{k-1})+1\\
&=s\circ s_{n_{2}}^{-1}(2^{k-1})+s\circ s_{n_{2}}^{-1}(1)\mbox{ since }s\circ s_{n_{2}}^{-1}(1)=1\\
&\leq s\circ s_{n_{2}}^{-1}(x)+s\circ s_{n_{2}}^{-1}(y).\\
\end{align*}
If $y=\infty$, then it is a trivial that $s\circ s_{n_{2}}^{-1}(x+y)=s\circ s_{n_{2}}^{-1}(x)+s\circ s_{n_{2}}^{-1}(y)$.

Therefore, $s\circ s_{n_{2}}^{-1}$ is subadditive. \\
\stepcounter{step} (\Roman{step}) Finally, we shall show that $s\circ s_{n_{1}}^{-1}$ is also subadditive.

By Eq.(\ref{eq:14}) we have
$$s\circ s_{n_{2}}^{-1}(2^{k})=\left(s\circ s_{n_{1}}^{-1}\right)\circ \left(s_{n_{1}}\circ s_{n_{2}}^{-1}\right)(2^{k})\leq s\circ s_{n_{2}}^{-1}(2^{k-1})\cdot \frac{s_{n_{1}}\circ s_{n_{2}}^{-1}(2^{k})}{s_{n_{1}}\circ s_{n_{2}}^{-1}(2^{k-1})},$$
or, equivalently,
$$\frac{\left(s\circ s_{n_{1}}^{-1}\right)\circ \left(s_{n_{1}}\circ s_{n_{2}}^{-1}\right)(2^{k})}{s_{n_{1}}\circ s_{n_{2}}^{-1}(2^{k})}\leq \frac{s\circ s_{n_{2}}^{-1}(2^{k-1})}{s_{n_{1}}\circ s_{n_{2}}^{-1}(2^{k-1})}=\frac{\left(s\circ s_{n_{1}}^{-1}\right)\circ \left(s_{n_{1}}\circ s_{n_{2}}^{-1}\right)(2^{k-1})}{s_{n_{1}}\circ s_{n_{2}}^{-1}(2^{k-1})}.$$
Let $a=\left(s_{n_{1}}\circ s_{n_{2}}^{-1}\right)(2^{k})$ and $b=\left(s_{n_{1}}\circ s_{n_{2}}^{-1}\right)(2^{k-1})$. Since $s_{n_{1}}\circ s_{n_{2}}^{-1}$ is strictly increasing, we have $a>b$ and
$$\frac{s\circ s_{n_{1}}^{-1}(a)}{a}\leq \frac{s\circ s_{n_{1}}^{-1}(b)}{b},$$
which together with Statement A implies that $\frac{s\circ s_{n_{1}}^{-1}(x)}{x}$ is non-increasing on $[1,\infty]$. Therefore, $s\circ s_{n_{1}}^{-1}$ is subadditive by Lemma \ref{lemm5.1}. \\
\end{proof}

By induction, the following corollary immediately follows from Theorem \ref{theorem:5.1}.
\begin{corollary}\label{coro5.1}
Let $S_{i}:[0,1]^2\rightarrow[0,1]$, $i=1,2,\ldots,n$, be a finite family of non-constant proper and strict t-subnorms with continuous, strictly decreasing additive generators. Then there exists a non-constant proper and strict t-subnorm $S_{l}:[0,1]^2\rightarrow[0,1]$ generated by a continuous, strictly decreasing additive generator such that
$$S_{l}\leq \min(S_{1}, S_{2},\ldots,S_{n}).$$
\end{corollary}

By duality, we have a completely analogous result as the following corollary.
\begin{corollary}\label{coro5.2}
Let $M_{i}:[0,1]^2\rightarrow[0,1]$, $i=1,2,\ldots,n$, be a finite family of non-constant proper and strict t-superconorms with continuous, strictly increasing additive generators. Then there exists a non-constant proper and strict t-superconorm $M_{u}:[0,1]^2\rightarrow[0,1]$ generated by a continuous, strictly increasing additive generator such that
$$M_{u}\geq \max(M_{1}, M_{2},\ldots,M_{n}).$$
\end{corollary}

Let $S_{1}:[0,1]^2\rightarrow[0,1]$ and $S_{2}:[0,1]^2\rightarrow[0,1]$ be two non-constant proper and strict t-subnorms with continuous, strictly decreasing additive generators. Then from the proof of Theorem \ref{theorem:5.1}, we can summarize the following algorithm for finding a lower bound of both $S_{1}$ and $S_{2}$.

\begin{algorithm}[H]
\caption{\textbf{Finding a lower bound for $S_{1}$ and $S_{2}$}}\label{alg:5.1}
\noindent
\begin{tabular}{p{0.15\linewidth}p{0.75\linewidth}}
\toprule
\textbf{Input} &  Two non-constant proper and strict t-subnorms $S_{1}, S_{2}$ with their respective continuous, strictly decreasing additive generators $s_1$ and $s_2$. \\
\midrule
\textbf{Output} & A non-constant proper and strict t-subnorms $S_{l}$ with a continuous, strictly decreasing additive generator $s_n$.\\
\midrule
\textbf{Step 1:} & Compute the normalized additive generators $s_{n_{1}}$ and $s_{n_{2}}$ of $S_{1}$ and $S_{2}$ by Eq.\eqref{eq5}, respectively. \\[5pt]
\textbf{Step 2:} & Compute the function $s_{n_{1}}\circ s_{n_{2}}^{-1}:[1,\infty]\rightarrow[1,\infty]$. \\
\textbf{Step 3:} & Compute $s_{n_{1}}\circ s_{n_{2}}^{-1}(2^{k})$, $k=0,1,2,\ldots.$ \\[5pt]
\textbf{Step 4:} & Compute $s\circ s_{n_{2}}^{-1}(2^{k})$, $k=0,1,2,\ldots$, defined by Eq.(\ref{eq:14}). \\
\textbf{Step 5:} & Compute $s\circ s_{n_{2}}^{-1}(x)$, $x\in[2^{k},2^{k+1}]$, $k=0,1,2,\ldots$,  defined by Eq.(\ref{eq:15}). \\
\textbf{Step 6:} & Compute the normalized additive generators $s:[0,1]\rightarrow[1,\infty]$ by $s\circ s_{n_{2}}^{-1}$ and $s_{n_{2}}$. \\
\textbf{Step 7:} & Compute $S_{l}:[0,1]^2\rightarrow[0,1]$ by $S_{l}(x,y)=s^{(-1)}(s(x) + s(y))$. \\
\textbf{Step 8:} & End. \\
\bottomrule
\end{tabular}
\end{algorithm}

\begin{remark}\label{remark:5.1} \emph{In Theorem \ref{theorem:5.1}, in general, $S_{l}\notin \{S_{1}, S_{2}\}$ even if $S_{1}\leq S_{2}$ as shown by the following example.}
\end{remark}

\begin{example}\label{exp5.1}\emph{Consider $S_{1}$ and $S_{2}$ in Example \ref{exp4.1}. \\[5pt]
Step1: The normalized additive generators are $s_{n_{1}}:[0,1]\rightarrow[1, \infty]$ with $s_{n_{1}}(x)=-\frac{\ln x}{\ln2}+1$ and $s_{n_{2}}:[0,1]\rightarrow[1, \infty]$ with $s_{n_{2}}(x)=\frac{2}{x}-1$, respectively. \\[5pt]
Step2: The function $\theta(x)\triangleq s_{n_{1}} \circ s_{n_{2}}^{-1}(x)=\frac{\ln (x+1)}{\ln2}$. \\[5pt]
Step3: $\theta(1)=1$, $\theta(2)=\frac{\ln3}{\ln2}$, $\theta(4)=\frac{\ln5}{\ln2}$, $\theta(8)=\frac{2\ln3}{\ln2}$, \ldots, $\theta(2^{k})=\frac{\ln(2^{k}+1)}{\ln2}$, $k=0,1,2,\ldots.$\\[5pt]
Step4: Compute $s\circ s_{n_{2}}^{-1}(2^{k})$, $k=0,1,2,\ldots$.
\begin{itemize}
  \item $k=0$, $s\circ s_{n_{2}}^{-1}(1)=1\triangleq p_{0}$;
  \item $k=1$, $s\circ s_{n_{2}}^{-1}(2)=\mbox{min}\left(\frac{\ln3}{\ln2}, \frac{3}{2}\right)=\frac{3}{2}\triangleq p_{1}$;
  \item $k=2$, $s\circ s_{n_{2}}^{-1}(4)=\mbox{min}\left(\frac{3}{2}\cdot \frac{\ln5}{\ln3}, \frac{11}{6}\right)=\frac{11}{6}\triangleq p_{2}$;
  \item $k=3$, $s\circ s_{n_{2}}^{-1}(8)=\mbox{min}\left(\frac{11}{3}\cdot \frac{\ln3}{\ln5},\frac{139}{66}\right)=\frac{139}{66}\triangleq p_{3}$;
  \item $\ldots$
  \item $k=n$, $s\circ s_{n_{2}}^{-1}(2^{n})=\mbox{min}\left(p_{n-1}\cdot \frac{\theta(2^{n})}{\theta(2^{n-1})}, p_{n-1}+\frac{1}{2p_{n-1}}\right)\triangleq p_{n}$.
\end{itemize}
Step5: Compute $s\circ s_{n_{2}}^{-1}(x)$, $x\in[2^{k},2^{k+1}]$, $k=0,1,2,\ldots$.
\begin{itemize}
  \item $k=0$, i.e., $x\in[1,2]$, $A_{0}=\frac{\frac{3}{2}-1}{\frac{\ln3}{\ln2}-1}=\frac{\ln2}{2\ln\frac{3}{2}}\triangleq q_{1}$, $B_{0}=p_0-q_{1}=1-\frac{\ln2}{2\ln\frac{3}{2}}$. Then $s\circ s_{n_{2}}^{-1}(x)=A_{0}\cdot \theta(x)+B_{0}=\frac{\ln\frac{x+1}{2}}{2\ln\frac{3}{2}}+1$ for any $x\in[1,2]$;
  \item $k=1$, i.e., $x\in[2,4]$, $A_{1}=\frac{\frac{11}{6}-\frac{3}{2}}{\frac{\ln5}{\ln2}-\frac{\ln3}{\ln2}}=\frac{\ln2}{3\ln\frac{5}{3}}\triangleq q_{2}$, $B_{1}=p_{1}-q_{2}\cdot \frac{\ln3}{\ln2}=\frac{3}{2}-\frac{\ln3}{3\ln\frac{5}{3}}$. Then $s\circ s_{n_{2}}^{-1}(x)=A_{1}\cdot \theta(x)+B_{1}=\frac{\ln\frac{x+1}{3}}{3\ln\frac{5}{3}}+\frac{3}{2}$ for any $x\in[2,4]$;
  \item $k=2$, i.e., $x\in[4,8]$, $A_{2}=\frac{\frac{139}{66}-\frac{11}{6}}{\frac{2\ln3}{\ln5}-\frac{\ln5}{\ln2}}=\frac{3\ln2}{11\ln\frac{9}{5}}\triangleq q_{3}$, $B_{2}=p_2-q_{3}\cdot \frac{\ln5}{\ln2}=\frac{11}{6}-\frac{3\ln5}{11\ln\frac{9}{5}}$. Then $s\circ s_{n_{2}}^{-1}(x)=A_{2}\cdot \theta(x)+B_{2}=\frac{3\ln\frac{x+1}{5}}{11\ln\frac{9}{5}}+\frac{11}{6}$ for any $x\in[4,8]$;
  \item $\ldots$
  \item $s\circ s_{n_{2}}^{-1}(x)=A_{k}\cdot \theta(x)+B_{k}$ for any $x\in[2^{k},2^{k+1}]$, $k=0,1,2,\ldots$.
\end{itemize}
Step6: Compute the normalized additive generator $s:[0,1]\rightarrow[1,\infty]$ (see Fig. \ref{fig6}). Let $y=s_{n_{2}}(x)=\frac{2}{x}-1$. Then $s(x)=\left(s\circ s_{n_{2}}^{-1}\right)(y)$.
\begin{itemize}
  \item When $y\in[1,2]$, i.e., $x\in\left[\frac{2}{3}, 1\right]$, $s(x)=\frac{\ln\frac{1}{x}}{2\ln\frac{3}{2}}+1$;
  \item When $y\in[2,4]$, i.e., $x\in\left[\frac{2}{5}, \frac{2}{3}\right]$, $s(x)=\frac{\ln\frac{2}{3x}}{3\ln\frac{5}{3}}+\frac{3}{2}$;
  \item When $y\in[4,8]$, i.e., $x\in\left[\frac{2}{9}, \frac{2}{5}\right]$, $s(x)=\frac{3\ln\frac{2}{5x}}{11\ln\frac{9}{5}}+\frac{11}{6}$;
  \item $\cdots$.
\end{itemize}
Step7: Compute $S_{l}$ by $S_{l}(x,y)=s^{(-1)}(s(x)+s(y))$ for any $x,y\in[0,1]$. By Theorem \ref{theorem:5.1}, $S_{l}\leq \min(S_{1}, S_{2})$.}

\emph{One can verify that $S_{l}\notin \{S_{1}, S_{2}\}$.}
\end{example}

\begin{figure}[H]
\centering
\includegraphics[width=0.8\textwidth, keepaspectratio]{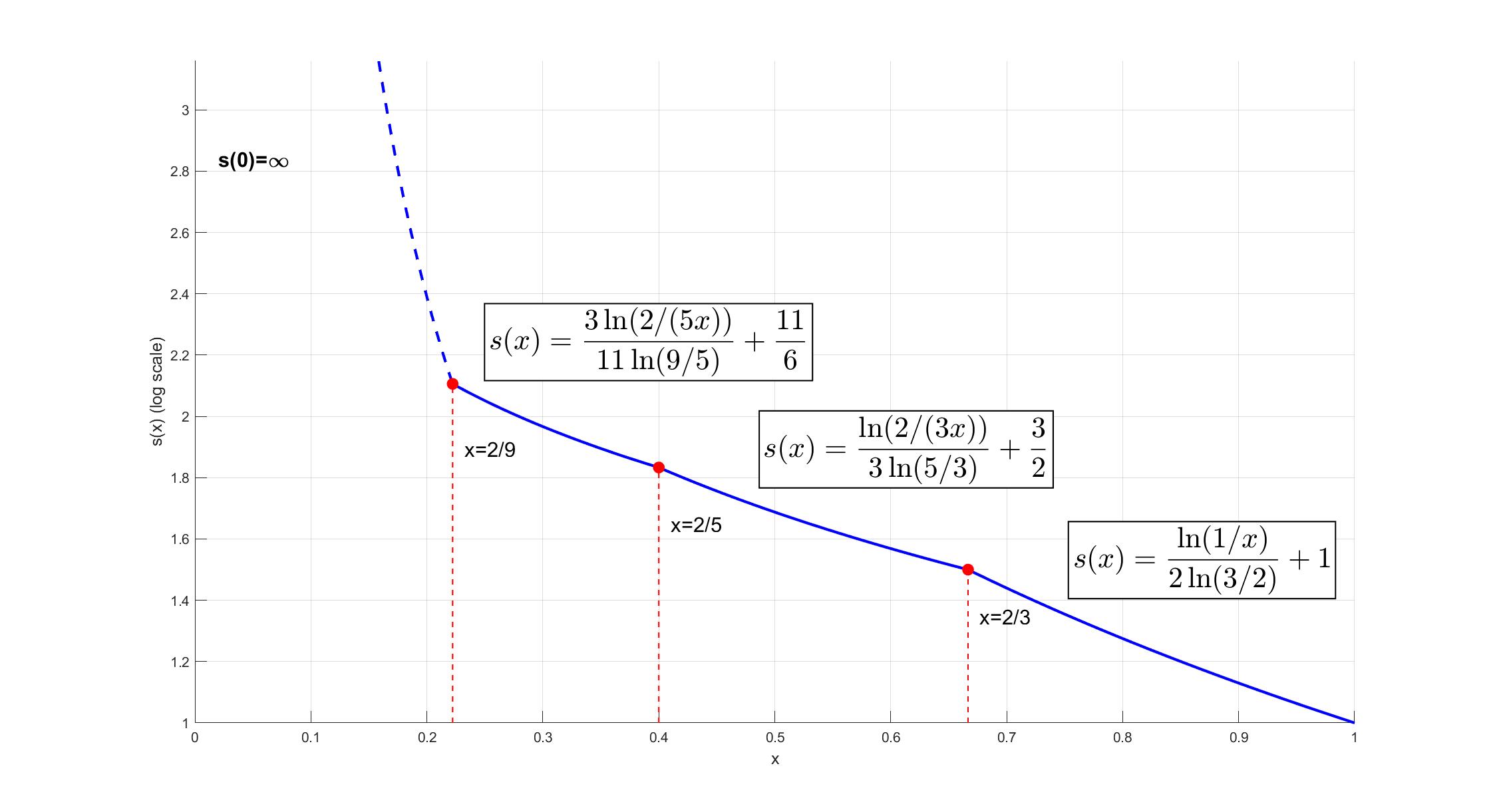}
\caption{Additive generator $s$ of $S_{l}$ in Example \ref{exp5.1}.}
\label{fig6}
\end{figure}

In summary, from Theorems \ref{theorem:4.1} and \ref{theorem:5.1} we immediately get the following corollary.
\begin{corollary}\label{coro5.3}
Let $S_{1}:[0,1]^2\rightarrow[0,1]$ and $S_{2}:[0,1]^2\rightarrow[0,1]$ be two non-constant proper and strict t-subnorms with continuous, strictly decreasing additive generators. Then there exist two non-constant proper and strict t-subnorms $S_{l}:[0,1]^2\rightarrow[0,1]$ and $S_{u}:[0,1]^2\rightarrow[0,1]$ generated by continuous, strictly decreasing additive generators such that
$$S_{l}\leq \min(S_{1}, S_{2})\leq \max(S_{1}, S_{2})\leq S_{u}.$$
\end{corollary}

Similar to Theorem \ref{thm4.1}, we have the following result.
\begin{theorem}[\cite{VM2001}]\label{thm4.15}
Let $T_{1}:[0,1]^2\rightarrow[0,1]$ and $T_{2}:[0,1]^2\rightarrow[0,1]$ be two continuous Archimedean t-norms. Then there exists a continuous Archimedean t-norm $T$ such that $T\leq \min(T_{1}, T_{2})$.
\end{theorem}

In conclusion, combining Theorems \ref{thm4.1} and \ref{thm4.15} and Corollary \ref{coro5.3}, we have the following theorem immediately.
\begin{theorem}\label{theorem:5.2}
Let $S_{1}:[0,1]^2\rightarrow[0,1]$ and $S_{2}:[0,1]^2\rightarrow[0,1]$ be two t-subnorms with continuous, strictly decreasing additive generators. If both $S_{1}$ and $S_{2}$ are strict, then there exist two strict t-subnorms $S_{l}:[0,1]^2\rightarrow[0,1]$ and $S_{u}:[0,1]^2\rightarrow[0,1]$ generated by continuous, strictly decreasing additive generators such that
$$S_{l}\leq \min(S_{1}, S_{2})\leq \max(S_{1}, S_{2})\leq S_{u}.$$
\end{theorem}

By duality, an analogous result holds for t-superconorms generated by continuous, strictly increasing additive generators as follows.
\begin{corollary}\label{coro5.4}
Let $M_{1}:[0,1]^2\rightarrow[0,1]$ and $M_{2}:[0,1]^2\rightarrow[0,1]$ be two t-superconorms with continuous, strictly increasing additive generators. If both $M_{1}$ and $M_{2}$ are strict, then there exist two strict t-superconorms $M_{l}:[0,1]^2\rightarrow[0,1]$ and $M_{u}:[0,1]^2\rightarrow[0,1]$ generated by continuous, strictly increasing additive generators such that
$$M_{l}\leq \min(M_{1}, M_{2})\leq \max(M_{1}, M_{2})\leq M_{u}.$$
\end{corollary}

\subsection{In the case of proper and nilpotent t-subnorms}
%In the case that both $S_{1}$ and $S_{2}$ are proper and nilpotent, we also completely give a lower bound for $S_{1}$ and $S_{2}$.

\begin{theorem}\label{theorem:5.3}
Let $S_{1}:[0,1]^2\rightarrow[0,1]$ and $S_{2}:[0,1]^2\rightarrow[0,1]$ be two non-constant proper and nilpotent t-subnorms with continuous, strictly decreasing additive generators. Then there exists a proper and nilpotent t-subnorm $S_{l}:[0,1]^2\rightarrow[0,1]$ generated by a continuous, strictly decreasing additive generator such that
$$S_{l}\leq \min(S_{1}, S_{2}).$$
\end{theorem}
\begin{proof}Since both $S_{1}$ and $S_{2}$ are proper and nilpotent, their normalized additive generators $s_{n_{i}}:[0,1]\rightarrow[1,s_{n_i}(0)]$ satisfy $1<s_{n_{i}}(0)<\infty$ for $i=1, 2$. By Theorem \ref{theorem:3.1}, it suffices to construct a strictly decreasing bijection $s:[0,1]\rightarrow[1,s(0)]$ such that both $s \circ s_{n_{1}}^{-1}$ and $s \circ s_{n_{2}}^{-1}$ are subadditive.

The construction follows precisely the proof of Theorem \ref{theorem:5.1}, with the only essential modification arising from the fact that the domain of $s_{n_{2}}$ is the finite interval $[1,s_{n_{2}}(0)]$ rather than $[1,\infty]$.

Let $k_{0}\in\mathbb{N}$ be the unique integer such that $2^{k_{0}} < s_{n_{2}}(0) \le 2^{k_{0}+1}$. We define the function $s\circ s_{n_{2}}^{-1}$ on $[1, s_{n_{2}}(0)]$ as follows:
\begin{itemize}
  \item Apply Eq.(\ref{eq:14}) to $k=1, 2, \dots, k_{0}$;
  \item $s\circ s_{n_{2}}^{-1}(s_{n_{2}}(0))=\mbox{min}\left(s\circ s_{n_{2}}^{-1}(2^{k_{0}})\cdot \frac{s_{n_{1}}\circ s_{n_{2}}^{-1}(s_{n_{2}}(0))}{s_{n_{1}}\circ s_{n_{2}}^{-1}(2^{k_{0}})}, s\circ s_{n_{2}}^{-1}(2^{k_{0}})+\frac{1}{2\cdot s\circ s_{n_{2}}^{-1}(2^{k_{0}})}\right)$;
\item  Apply Eq.(\ref{eq:15}) to each subinterval $[2^k,2^{k+1}]$ for $k=0, 1, \dots, k_{0}-1$, and on the final subinterval $[2^{k_{0}}, s_{n_{2}}(0)]$ we also use Eq.(\ref{eq:15}), just noting that we have to replace the respective $2^k$and $2^{k+1}$ by $2^{k_{0}}$ and $s_{n_{2}}(0)$.
\end{itemize}

Then $s\circ s_{n_{2}}^{-1}:[1,s_{n_{2}}(0)]\rightarrow[1,s(0)]$ is a continuous, strictly increasing bijection. Consequently, $s=(s \circ s_{n_{2}}^{-1}) \circ s_{n_{2}}:[0,1]\rightarrow[1, s(0)]$ is a continuous, strictly decreasing bijection.

The verification of subadditivity for both $s \circ s_{n_{2}}^{-1}$ and $s \circ s_{n_{1}}^{-1}$ proceeds exactly as in the proof of Theorem \ref{theorem:5.1}. This completes the proof.
\end{proof}

From the proof of Theorem \ref{theorem:5.3}, we can also derive an algorithm similar to Algorithm \ref{alg:5.1}, which is used to find a nilpotent lower bound of two proper nilpotent t-subnorms. Here we omit it.
\begin{remark}\label{remark:5.1} \emph{Sometimes, the t-subnorm $S_{l}$ obtained by Theorem \ref{theorem:5.3} may be constant, i.e., $S_{l}=Z$, as illustrated by the following example.}
\end{remark}

\begin{example}\label{exp5.2}\emph{Let $S_{1}:[0,1]^2\rightarrow[0,1]$ with $S_{1}(x,y)=\mbox{max}\left(2-\sqrt{(2-x)^{2}+(2-y)^{2}},0\right)$ and $S_{2}:[0,1]^2\rightarrow[0,1]$ with $S_{2}(x,y)=\mbox{max}\left(x+y-\frac{3}{2},0\right)$ generated by their continuous, strictly decreasing additive generators $s_{1}:[0,1]\rightarrow[\frac{1}{4},1]$ with $s_{1}(x)=(1-\frac{x}{2})^{2}$ and $s_{2}:[0,1]\rightarrow[\frac{1}{3},1]$ with $s_{2}(x)=1-\frac{2}{3}x$. Obviously,  both $S_{1}$ and $S_{2}$ are proper and nilpotent t-subnorms (see Fig. \ref{fig7}). \\[5pt]
Step1: The normalized additive generators are $s_{n_{1}}:[0,1]\rightarrow[1,4]$ with $s_{n_{1}}(x)=(2-x)^{2}$ and $s_{n_{2}}:[0,1]\rightarrow[1,3]$ with $s_{n_{2}}(x)=3-2x$, respectively. \\[5pt]
Step2: The function $\theta(x)\triangleq s_{n_{1}} \circ s_{n_{2}}^{-1}(x)=\frac{(1+x)^{2}}{4}$, and clearly $\theta$ is a function from $[1,3]$ to $[1,4]$. \\[5pt]
Step3: $\theta(1)=1$, $\theta(2)=\frac{9}{4}$, $\theta(3)=4$. \\[5pt]
Step4: Compute $s\circ s_{n_{2}}^{-1}(2^{k})$, where $k=0,1$, and $s\circ s_{n_{2}}^{-1}(3)$.
\begin{itemize}
  \item $k=0$, $s\circ s_{n_{2}}^{-1}(1)=1\triangleq p_{0}$;
  \item $k=1$, $s\circ s_{n_{2}}^{-1}(2)=\mbox{min}\left(\frac{9}{4}, \frac{3}{2}\right)=\frac{3}{2}\triangleq p_{1}$;
  \item $s\circ s_{n_{2}}^{-1}(3)=\mbox{min}\left(\frac{8}{3}, \frac{11}{6}\right)=\frac{11}{6}\triangleq p_{2}$.
  \end{itemize}
Step5: Compute $s\circ s_{n_{2}}^{-1}(x)$, $x\in[2^{k},2^{k+1}]$, $k=0,1$.
\begin{itemize}
  \item $k=0$, i.e., $x\in[1,2]$, $A_{0}=\frac{\frac{3}{2}-1}{\frac{9}{4}-1}=\frac{2}{5}\triangleq q_{1}$, $B_{0}=p_{0}-q_{1}=\frac{3}{5}$. Then $s\circ s_{n_{2}}^{-1}(x)=A_{0}\cdot \theta(x)+B_{0}=\frac{(1+x)^{2}}{10}+\frac{3}{5}$ for any $x\in[1,2]$;
  \item $k=1$, i.e., $x\in[2,3]$, $A_{1}=\frac{\frac{11}{6}-\frac{3}{2}}{4-\frac{9}{4}}=\frac{4}{21}\triangleq q_{2}$, $B_{1}=p_{1}-q_{2}\cdot \frac{9}{4}=\frac{15}{14}$. Then $s\circ s_{n_{2}}^{-1}(x)=A_{1}\cdot \theta(x)+B_{1}=\frac{(1+x)^{2}}{21}+\frac{15}{14}$ for any $x\in[2,3]$.
  \end{itemize}
Step6: Compute the normalized additive generator $s:[0,1]\rightarrow[1,\frac{11}{6}]$ (see Fig. \ref{fig8}). Let $y=s_{n_{2}}(x)=3-2x$. Then $s(x)=\left(s\circ s_{n_{2}}^{-1}\right)(y)$.
\begin{itemize}
  \item When $y\in[1,2]$, i.e., $x\in[\frac{1}{2},1]$, $s(x)=\frac{2}{5}(2-x)^{2}+\frac{3}{5}$;
  \item When $y\in[2,3]$, i.e., $x\in[0,\frac{1}{2}]$, $s(x)=\frac{4}{21}(2-x)^{2}+\frac{15}{14}$.
\end{itemize}
 Since $s(0)=\frac{11}{6}<2=2s(1)$, we have $S_{l}=Z$ by Proposition \ref{pro2.1}(iv).}
\end{example}

\begin{figure}[H]
\centering
\includegraphics[width=1.0\textwidth, keepaspectratio]{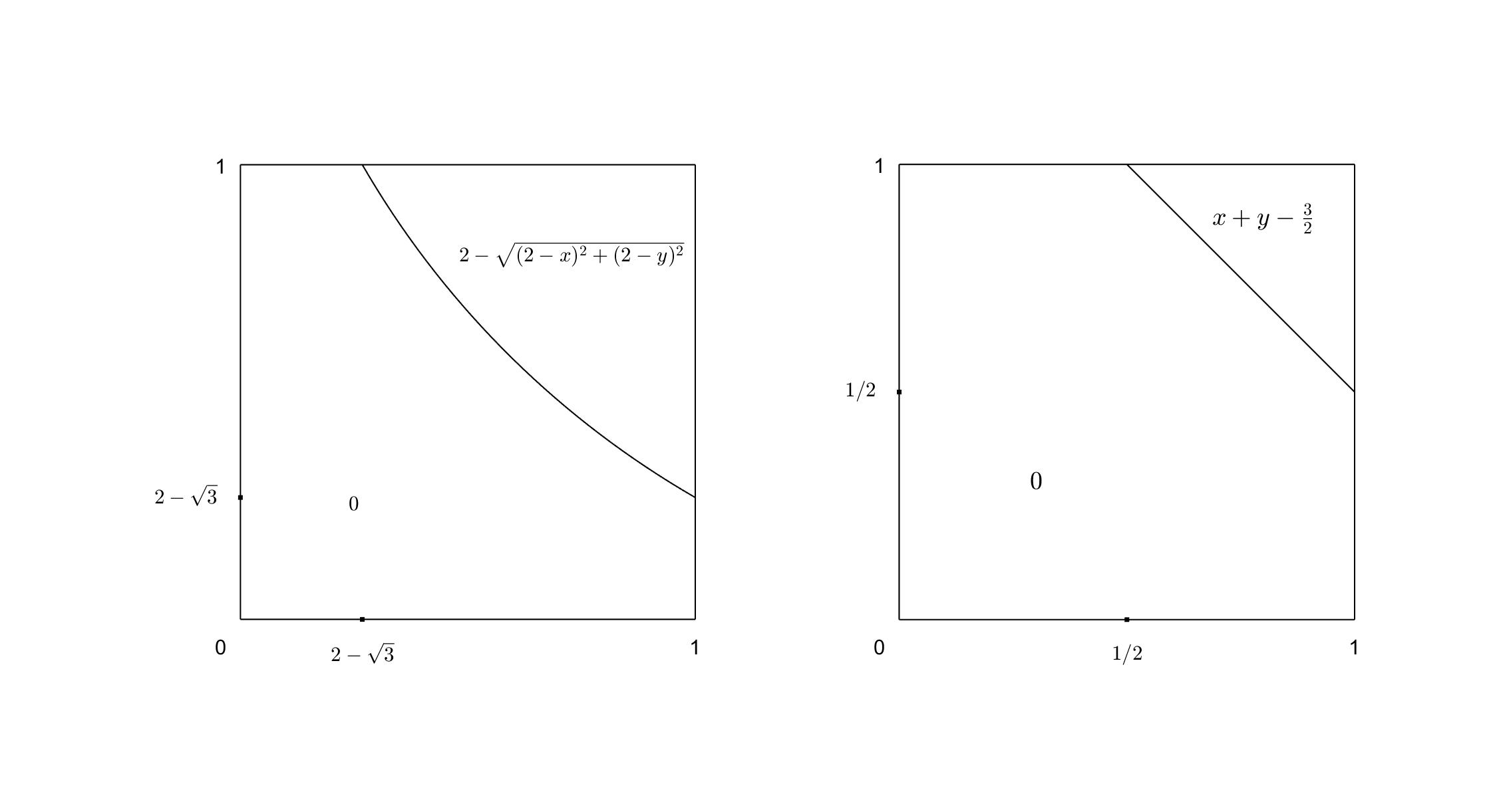}
\caption{two proper and nilpotent t-subnorms in Example \ref{exp5.2}.}
\label{fig7}
\end{figure}

\begin{figure}[H]
\centering
\includegraphics[width=0.8\textwidth, keepaspectratio]{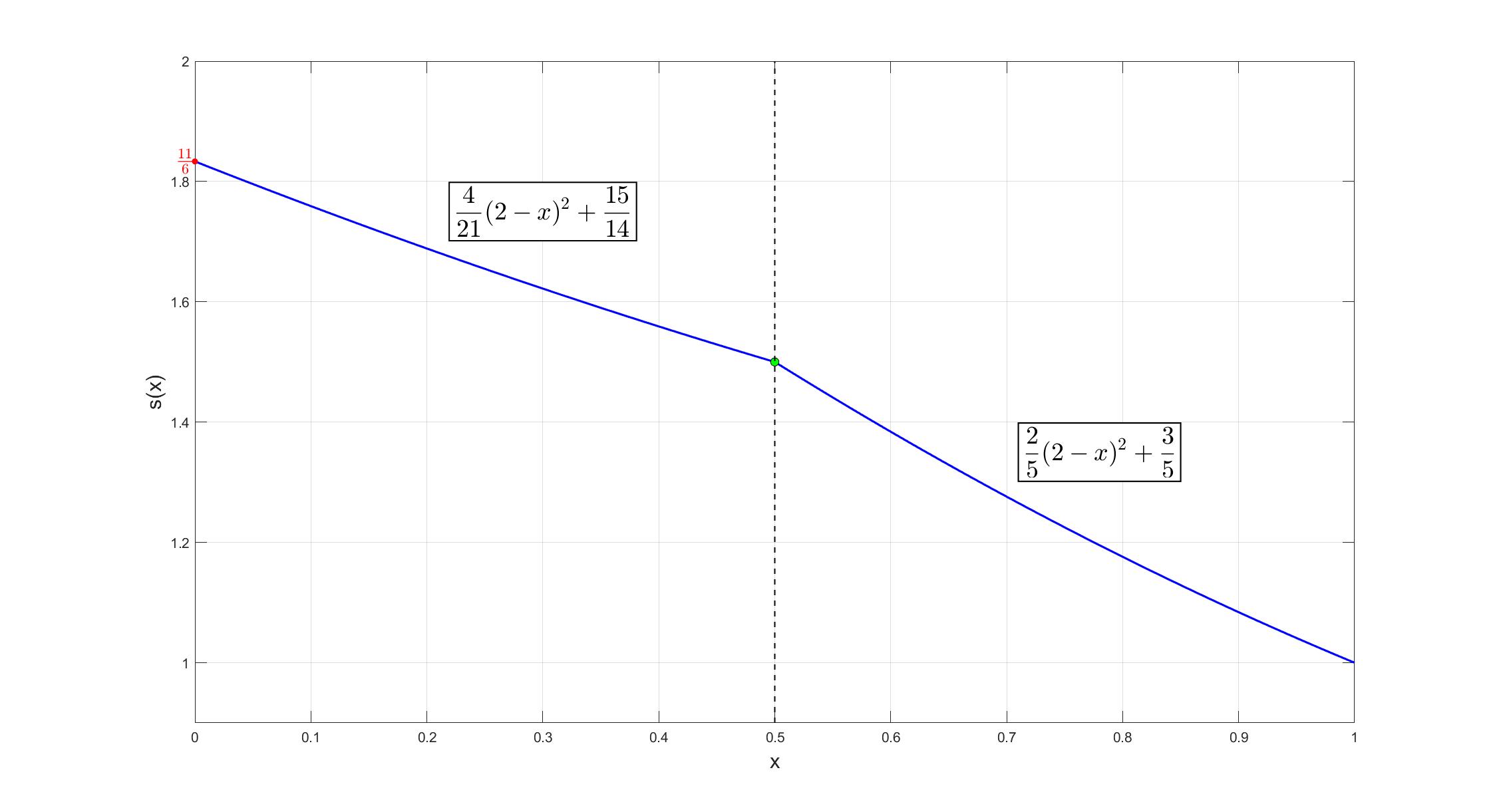}
\caption{Additive generator $s$ of $S_{l}$ in Example \ref{exp5.2}.}
\label{fig8}
\end{figure}

By induction, the following corollary is an immediate consequence of Theorem \ref{theorem:5.3}.
\begin{corollary}\label{coro5.5}
Let $S_{i}:[0,1]^2\rightarrow[0,1]$, $i=1,2,\ldots,n$, be a finite family of non-constant proper and nilpotent t-subnorms with continuous, strictly decreasing additive generators. Then there exists a proper and nilpotent t-subnorm $S_{l}:[0,1]^2\rightarrow[0,1]$ generated by a continuous, strictly decreasing additive generator such that
$$S_{l}\leq \min(S_{1}, S_{2},\ldots,S_{n}).$$
\end{corollary}

Furthermore, the lower bounds for proper t-subnorms with continuous, strictly decreasing additive generators follow simply from Theorem \ref{theorem:5.1}, Theorem \ref{theorem:5.3} and Corollary \ref{coro5.5} as follows.
\begin{proposition}\label{pro5.1}
Let $S_{i}:[0,1]^2\rightarrow[0,1]$, $i=1,2,\ldots,n$, be a finite family of non-constant proper t-subnorms with continuous, strictly decreasing additive generators. Then there exists a proper t-subnorm $S_{l}:[0,1]^2\rightarrow[0,1]$ generated by a continuous, strictly decreasing additive generator such that
$$S_{l}\leq \min(S_{1}, S_{2},\ldots,S_{n}).$$
\end{proposition}

Duality yields a fully analogous result on the upper bounds of t-superconorms generated by continuous, strictly increasing additive generators, as the following corollary.
\begin{corollary}\label{coro5.6}
Let $M_{i}:[0,1]^2\rightarrow[0,1]$, $i=1,2,\ldots,n$, be a finite family of non-constant proper t-superconorms with continuous, strictly increasing additive generators. Then there exists a proper t-superconorm $M_{u}$ generated by a continuous, strictly increasing additive generator such that
$$M_{u}\geq \max(M_{1}, M_{2},\ldots,M_{n}).$$
\end{corollary}

\section{Conclusions}
 This article paid attention to the upper and lower bounds of t-subnorms generated by continuous, strictly decreasing additive generators. It not only established necessary and sufficient conditions for the comparison of two t-subnorms through their additive generators but also demonstrated the existence of strict (resp. nilpotent) upper and lower bounds of a finite family of strict (resp. nilpotent) t-subnorms generated by continuous, strictly decreasing additive generators. By duality, completely analogous results are derived for the strict (resp. nilpotent) t-superconorms generated by continuous, strictly increasing additive generators. In addition, this article provided the corresponding algorithms for computing those bounds, which offer practical tools for both theoretical analysis and applications in fields such as fuzzy logic and information aggregation. It is worth noting that our results are suitable for t-subnorms (resp. t-superconorms) generated by continuous, strictly decreasing (resp. increasing) additive generators. A more general problem is how to find the upper and lower bounds of t-subnorms (resp. t-superconorms) within class of t-subnorms (resp. t-superconorms) without such restrictions on additive generators.

%\section*{Acknowledgments}
%The authors would like to thank the anonymous referees for their valuable comments and suggestions.

\end{document}